\documentclass[10pt]{article}

\usepackage{amssymb}
\usepackage{amsfonts}
\usepackage{graphicx}
\usepackage{amsmath}

\usepackage{appendix}

\usepackage[colorlinks, linkcolor=black, anchorcolor=black,citecolor=black]{hyperref}

\usepackage{mathrsfs}
\usepackage{frcursive}

\usepackage{geometry}
\usepackage{color}

\numberwithin{equation}{section}

\def\[{\left[}
\def\]{\right]}
\def\({\left(}
\def\){\right)}

\renewcommand{\d}{\mathrm{d}}
\newcommand{\e}{\mathrm{e}}

\newcommand{\R}{\mathbb{R}}

\newtheorem{theorem}{Theorem}[section]

\newtheorem{assumption}[theorem]{Assumption}

\newtheorem{claim}[theorem]{Claim}

\newtheorem{corollary}[theorem]{Corollary}

\newtheorem{definition}[theorem]{Definition}

\newtheorem{lemma}[theorem]{Lemma}

\newtheorem{proposition}[theorem]{Proposition}
\newtheorem{remark}[theorem]{Remark}

\newenvironment{proof}[1][Proof]{\noindent\textit{#1.} }{\hfill \rule{0.5em}{0.5em}}

\renewcommand{\eqref}[1]{(\ref{#1})}

\begin{document}

\title{\textsc{Spreading properties for non-autonomous Fisher-KPP equations with nonlocal diffusion}}

%\author{\textsc{Arnaud Ducrot$^a$ and Zhucheng Jin$^{b}$}\\
%    Normandie Univ., UNIHAVRE, LMAH, FR-CNRS-3335,\\
%	ISCN, 76600 Le Havre, France\\
%	$^a$ arnaud.ducrot@univ-lehavre.fr\\
%	$^b$ zhucheng.jin@univ-lehavre.fr}
%	

\author{\textsc{Arnaud Ducrot$^{a,}$\footnote{Corresponding author.\\
  \indent E-mail: arnaud.ducrot@univ-lehavre.fr (Arnaud Ducrot), jinzc@ustc.edu.cn (Zhucheng Jin).}  and  Zhucheng Jin$^{a,b}$} \\
$^{a}$Normandie Univ., UNIHAVRE, LMAH, FR-CNRS-3335,\\
 ISCN, 76600 Le Havre, France\\
{ $^{b}$School of Mathematical Sciences, University of Science and Technology of China, } \\
Hefei, Anhui 230026, China 
}

	\maketitle	

%
%This is a joint work with Arnaud Ducrot, submitted \cite{ducrot2022spreading}.

\begin{abstract}
We investigate the large time behaviour of solutions to a non-autonomous Fisher-KPP equation with nonlocal diffusion, involving a thin-tailed kernel. In this paper, we are concerned with both compactly supported and exponentially decaying initial data. 
As far as general time heterogeneities are concerned, we provide upper and lower estimates for the location of the propagating front. 
As a special case, we derive a definite spreading speed when the time varying coefficients satisfy some averaging properties. This setting covers the cases of periodic, almost periodic and uniquely ergodic variations in time, in particular.
Our analysis is based on the derivation of suitable regularity estimates (of uniform continuity type) for some particular solutions of a logistic equation with nonlocal diffusion. Such regularity estimates are coupled with the construction of appropriated propagating paths to derive spreading speed estimates, using ideas from the uniform persistence theory in dynamical systems.

		\vspace{0.2in}\noindent \textbf{Key words}. Spreading speeds; Nonlocal diffusion; Time heterogeneity; Fisher-KPP equation; Uniform persistence. 
		
		\vspace{0.1in}\noindent \textbf{2010 Mathematical Subject Classification}.	35B40; 45K05; 35C07.\\

\vspace{1ex}
		
		\noindent There is no conflict of interests.\\
                  There is no data associated to this work.\\
		
	\end{abstract}

\section{Introduction and main results}	
In this paper we study spreading properties for the  solutions of the following non-autonomous  and nonlocal one-dimensional equation
\begin{equation}\label{Pb-dj2}
	\partial_t u(t,x)= \int_{\R} K(y) \left[ u(t,x-y) -u(t,x) \right] \d y +  u(t,x)f \left(t, u(t,x)\right),
\end{equation}
posed for time $t\geq 0$ and $x\in \R$. This evolution problem is supplemented with an appropriated initial data, that will be discussed below. Here $K=K(y)$ is a nonnegative dispersal kernel with thin-tailed (see Assumption \ref{ASS1-dj2} below).  {\color{red} Let us set $F(t,u):= uf(t,u)$.}   At the same time, $F=F(t,u)$ stands for the nonlinear growth term, which depends on time $t$ and that will be assumed in this note to be of the Fisher-KPP type (see Assumption \ref{ASS2-dj2}). The above problem typically describes the spatial invasion of a population (see for instance \cite{berestycki2016persistence, lutscher2005effect} and the references therein) with the following features:\\
1) individuals exhibit long distance dispersal according to the kernel $K$, in other words the quantity $K(x-y)$ corresponds to the probability for individuals to jump from $y$ to $x$;\\
2) time varying birth and death processes modeled by the nonlinear Fisher-KPP type function $F(t,u)$. The time variations may stand for seasonality and/or external events (see \cite{jin2012seasonal}).

\vspace{1ex}

When local diffusion is considered, the Fisher-KPP equation posed in a time homogeneous medium reads as
\begin{equation}\label{Fisher-kPP}
	\partial_t u(t,x)= \partial_{xx} u(t,x) + F(u(t,x)).
\end{equation}
As mentioned above, this problem arises as a basic model in many different fields, in biology and ecology in particular. It can be used for instance to describe the spatio-temporal evolution of an invading species into an empty environment.
The above equation \eqref{Fisher-kPP} was introduced separately by Fisher \cite{fisher1937wave} and Kolmogorov, Petrovsky and Piskunov \cite{kolmogorov1937study}, when the nonlinear function $F$ satisfies the Fisher-KPP conditions. Recall that a typical example of such Fisher-KPP nonlinearity is given by the logistic function $F(u)= u(1-u)$.

There is a large amount of literature related to \eqref{Fisher-kPP} and its generalizations. To study propagation phenomena generated by reaction diffusion equations, in addition to the existence of travelling wave solution, the asymptotic speed of spread (or spreading speed) was introduced and studied by Aronson and Weinberger in \cite{aronson1978multidimensional}. Roughly speaking if $u_0$ is a nontrival and nonnegative initial data with compact support, then the solution of \eqref{Fisher-kPP} associated with this initial data $u_0$ spreads with the speed $c^*>0$ (the minimal wave speed of the travelling waves) in the sense that
$$
\lim_{t\to \infty} \sup_{|x|\leq c t} | u(t,x)-1|=0, \; \forall \; c \in [0,c^*)  \; \text{ and } \; \lim_{t\to \infty} \sup_{|x|\geq c t} u(t,x)=0, \; \forall  \; c>c^*.
$$

This concept of spreading speed has been further developed by several researchers in the last decades from different view points including PDE's argument, dynamical systems theory, probability theory, mathematical biology, etc.  
Spreading speeds of KPP-type reaction diffusion equations in homogeneous and periodic media have been extensively studied (see \cite{berestycki2005speed,  fang2017traveling, liang2006spreading, liang2007asymptotic,  weinberger1982long,  weinberger2002spreading} and the references cited therein).  There is also an extensive literature on spreading phenomena for reaction diffusion systems. We refer for instance \cite{ ambrosio2021generalized, ducrot2019spreading, girardin2019invasion} and  the references cited therein.

Recently spreading properties for KPP-type reaction-diffusion equations in more  general environments have attracted a lot of attention, see 
\cite{berestycki2008asymptotic, berestycki2019asymptotic, nadin2012propagation, shen2010variational} and the references cited therein.   
In particular, Nadin and Rossi \cite{nadin2012propagation} studied spreading properties for KPP equation with local diffusion and general time heterogeneities. Furthermore, they obtained a definite spreading speed when the coefficients {\color{red} share some averaging properties.}

The spreading properties of nonlocal diffusion equation as \eqref{Pb-dj2} has attracted a lot of interest in the last decades. Since the semiflow generated by nonlocal diffusion equations does not enjoy any regularization effects, this brings additional difficulties. Fisher-KPP equations or monostable problems in homogeneous environments have been studied from various point of views: wave front propagation (see \cite{coville2005propagation, schumacher1980travelling} and the references cited therein), hair trigger effect and spreading speed (see \cite{ alfaro2017fujita, cabre2013influence, diekmann1979run, finkelshtein2018hair, lutscher2005effect, xu2021spatial} and the references cited therein). For the thin-tailed kernel, we refer for instance to \cite{lutscher2005effect} and the recent work \cite{xu2021spatial} where a new sub-solution has been constructed to provide a lower bound of the spreading speed. Note also that the aforementioned work deals with possibly non-symmetric kernel where the propagation speed on the left and the right-hand side of the domain can be different. 
For the fat-tailed dispersion kernels the propagating behaviour of the solutions can be very different from the one observed with thin-tailed kernel. Acceleration may occur. We refer to \cite{finkelshtein2019accelerated, garnier2011accelerating} for fat-tailed kernel and to \cite{cabre2013influence} for fractional Laplace type dispersion.

Recently, wave propagation and spreading speeds for nonlocal diffusion problem with  time and/or space heterogeneities have been considered. Existence and nonexistence of generalized travelling wave solutions have been discussed in \cite{ducrot2021generalized, jin2009spatial, lim2016transition,  shen2016transition} and the references cited therein. For spreading speed results, we refer the reader to \cite{jin2012seasonal, jin2009spatial, liang2020jfa, shen2010spreading} and the references cited therein.
We also refer to \cite{bao2018spreading, xu2020spatial, zhang2020propagation} for the analysis of the spreading speed for systems with nonlocal diffusion.

As far as monotone problem is concerned, one may apply the well developed monotone semiflow  method  to study the spreading speed for nonlocal diffusion problems. We refer the reader to \cite{liang2007asymptotic, weinberger1982long} and to  \cite{jin2012seasonal, jin2009spatial} for time periodic systems.\\
In this work, we provide a new approach which is based on the
construction of suitable propagating paths (namely, functions   $t\mapsto X(t)$  with  $\liminf_{t\to \infty} u(t,X(t))>0$) coupled with what we call a persistence lemma (see Lemma \ref{LEM-key-dj2} below) for uniformly continuous solutions, to obtain lower estimate for the propagating set. This lemma roughly states that controlling from below the solution at $x=0$ and $X(t)$ for $t\gg 1$ implies a control of the solution $u=u(t,x)$ from below on the whole interval $x\in [0,kX(t)]$ for some $k\in (0,1)$ and $t\gg 1$. 
The proof of this lemma does not make use of the properties of the tail of the kernel, so that we expect our key persistence lemma to be applied for the study of acceleration phenomena for fat tailed dispersal kernel.  However, the uniform continuity property for the solutions is important for our proof and this remains complicated to check. 
{\color{red} For the regularity results  of some specific time global solutions to nonlocal diffusion equations, we refer the reader to  \cite{coville2013pulsating, lim2016transition}  for spatial heterogeneous case and to \cite{ducrot2021generalized, shen2016transition} for time heterogeneous media. }
Here we are able to prove such a property for some specific initial data and logistic type nonlinearities.

%To our knowledge, there is no Harnack's inequality for parabolic nonlocal diffusion equation and herein the proof of regularity differs from studying travelling wave which is sufficiently to show the uniformly continuity for a special class of solutions whose initial data is front like (see Lemma 5.1 in \cite{DJ1} and  Lemma 4.1 in \cite{LZ2016}). 
Note that in \cite{li2010entire} the authors consider the  regularity problem. They show that when the nonlinear term satisfies $F_u(u) < \overline{K}$ for any $u\geq 0$, where $\overline{K}= \int_{\R}K(y)\d y$, then  solutions of the homogeneous problem inherit the Lipschitz continuity property from those of their initial data, with a control of the Lipschitz constant for all time $t\gg 1$.
In this note, we prove the uniform continuity of some solutions when the above condition fails (see Assumption \ref{ASS2-dj2} $(f4)$).  
This point is studied in Section \ref{sectionkey}, where we provide a class of initial data for which the solutions (of the nonlocal logistic equation) are uniformly continuous on $[0,\infty)\times \mathbb R$. 

Now to state our results, we first introduce some notations and present our main assumptions.
Let us define the important notion of the least mean for a bounded function.
\begin{definition}\label{DEF-least-dj2}
	Along this work, for any given function $h\in L^\infty(0, \infty;\R)$,  we define
	\begin{equation}\label{def_least}
		\lfloor h \rfloor := \lim\limits_{T\to +\infty}\inf_{s> 0} \frac{1}{T} \int_{0}^{T} h(t+s) \d t.
	\end{equation} 
	In that case the quantity $\lfloor h \rfloor$ is called the \textbf{least mean} of the function $h$ (over $(0,\infty)$). 
\end{definition}

If $h$ admits a \textit{mean value} $\langle h \rangle $, that is, there exists 
\begin{equation}\label{mean}
	\langle h \rangle := \lim\limits_{T\to +\infty} \frac{1}{T} \int_{0}^{T} h(t+s) \d t, \text{ uniformly with respect to } s\geq 0.
\end{equation}
Then $\lfloor h \rfloor = \langle h \rangle $.  Particularly, the time periodic, almost periodic and uniquely ergodic coefficients have the mean value.  {\color{red}    Here recall that a bounded and uniformly continuous function $f:\R\to \R$ is called \textbf{\textit{uniquely ergodic}} if, for any continuous map $G: \mathcal{H}_{f} \to \R$, the following limit exists uniformly in $s\in \R$:   
$$  \lim\limits_{T\to +\infty} \frac{1}{T} \int_{s}^{s+T} G(f(\cdot+\tau) )\d \tau,  $$
where $\mathcal{H}_{f} := {\rm cl}\{f(\cdot + \tau),  \; \tau \in \R \}$ is the closure of the translation set of $f$ under the local  uniform topology.

 Periodic, almost periodic  and compactly supported functions are specific  subclass of uniquely ergodic functions.  A celebrated example of uniquely ergodic function is constructed from the Penrose tiling.
For more examples and properties of  almost periodic and uniquely ergodic functions, we refer the reader to \cite{berestycki2019asymptotic, lou2013recurrent, matano2011large}. }

An equivalent and useful characterization for the least mean  of  the function, as above, is given in the next lemma.
\begin{lemma}\cite{nadin2012propagation, nadin2015transition} \label{LEM-least}
	Let $h\in L^\infty(0,\infty;\R)$ be given. Then one has 
	$$
	\lfloor h \rfloor =\sup_{a\in W^{1, \infty}(0, \infty) }\inf_{t>0} \left( a' + h\right) (t).
	$$
	
\end{lemma}

We are now able to present the main assumptions that will be needed in this note.
First we assume that the kernel $K=K(y)$ enjoys the following set of properties:

\begin{assumption}[Kernel $K=K(y)$]\label{ASS1-dj2}
	We assume that the kernel $K: \R\to [0,\infty)$ satisfies the following set of assumptions:
	\begin{itemize}
		\item[(i)] The function $y\mapsto K(y)$ is non-negative, continuous and integrable;
		\item[(ii)] There exists $\alpha>0$ such that 
		$$
		\int_\R K(y)e^{\alpha y}dy<\infty.
		$$
		\item[(iii)] We also assume that $K(0)>0$.
	\end{itemize}
\end{assumption}

\begin{remark}\label{remark-k}
	Here we do not impose that the kernel function is symmetric.  We focus on the propagation to the right-hand side of the spatial domain.
	Thus in $(ii)$, we only assume the kernel  is thin-tailed on the right-hand side. 
	
	 Since  $K(y)$ is continuous and $K(0)>0$, then there exist $\delta >0$ and $k:\R\to [0,\infty)$, continuous, even and compactly supported such that
	\begin{equation}\label{lower-k}
		\begin{split}
			&{\rm supp}\;k=[-\delta,\delta],\;k(y)>0,\;\forall y\in (-\delta,\delta),\\
			&k(y)\leq K(y) \text{ and } k(y)=k(-y), \; \forall y\in\R.
		\end{split}
	\end{equation}
This property will allow us to control the solution on bounded sets, around $x=0$.	

\end{remark}

Now we discuss our Fisher-KPP assumptions for the nonlinear term $F(t,u)=uf(t,u)$. 	
\begin{assumption}[KPP nonlinearity]\label{ASS2-dj2}
	Assume that  the function $f:[0, \infty)\times [0,1]\to \R$ satisfies the following set of hypotheses:
	\begin{itemize}
		\item[(f1)]  $f(\cdot, u)\in L^\infty(0, \infty;\R)$, for all $u\in [0,1]$, and $f$ is Lipschitz continuous with respect to $u\in [0,1]$, uniformly with respect to $t\geq 0$;
		\item[(f2)] Let $f(t,1)=0$ for a.e. $t\geq 0 $. Setting $\mu(t):=f(t,0)$, we assume that $\mu(\cdot)$ is bounded and uniformly continuous. Also, we require that
		\begin{equation*}	
			h(u):=\inf_{t\geq 0} f(t,u)>0 \text{ for all $u\in [0,1)$};
		\end{equation*}
		\item[(f3)] For almost every $t\geq 0$, the function $u\mapsto f(t,u)$ is nonincreasing on $[0,1]$;
		\item[(f4)] Set $\overline{K}:= \int_\R K(y)\d y$. The least mean of the function $\mu$ satisfies
		$$\lfloor \mu \rfloor >\overline{K}. $$
	\end{itemize}
\end{assumption}

\begin{remark}
Here we assume that the steady states  are $p^-= 0$ and $p^+ =1$. These assumptions can be relaxed by the change of variables to take into account $p^-=p^-(t)$ and $p^+= p^+(t)$. Indeed, under the conditions $\inf_{t\geq 0} p^+(t) - p^-(t)>0$ and $p^+(t) - p^-(t)$ is bounded, one can  set 
$$ \tilde{u}(t,x) := \frac{u(t,x) -p^-(t)}{p^+(t) -p^-(t)}.$$
This can reduce the equation heterogeneous steady states into the equation with steady states $0$ and $1$ as long as $\inf_{t\geq 0} p^+(t) - p^-(t)>0$ and $p^+(t) - p^-(t)$ is bounded.  
\end{remark}
\begin{remark}\label{remark-f-dj2}
	From the above assumption, one can note that
	$$
	\inf_{t\geq 0}\mu(t)=h(0)>0.
	$$
	Next this assumption also implies that there exists some constant $C>0$ such that for all $u\in [0,1]$ and $t\geq 0$ one has
	\begin{equation}\label{lip-dj2}
		\mu(t)\geq f(t,u)\geq \mu(t) - Cu \geq \mu(t)(1- Hu),
	\end{equation}
	where we have set $H:= \sup \limits_{t\geq 0} \frac{C}{\mu(t)}=\frac{C}{h(0)}$.
%	
%	Assumption $(f4)$ is imposed for  some technical reasons. It will be used when we prove the hair trigger effect property in \eqref{Pb-dj2}. 
\end{remark}	
\medskip

Let us now define some notations related to the speed function that will be used in the following. 
We define $\sigma(K)$, the abscissa of convergence of $K$, by 
$$
\sigma\left( K\right):=\sup\left\{\gamma>0:\;\int_\R K(y)e^{\gamma y}dy<\infty\right\}. 
$$
Assumption \ref{ASS1-dj2} $(ii)$ yields that $\sigma(K)\in (0,\infty]$.
We set 
\begin{equation}\label{DEF-L-dj2}
	L(\lambda):= \int_{\R} K( y) [e^{\lambda y} -1 ] \d y, \; \lambda \in  \left[0, \sigma(K) \right),
\end{equation}
as well for $\lambda \in (0, \sigma(K))$ and $t\geq 0$,
\begin{equation}\label{DEF-c}
	c(\lambda)(t) := \lambda^{-1}L(\lambda)+\lambda^{-1}\mu(t).
\end{equation}
For a given function $a\in W^{1, \infty}(0,\infty)$,   denote $c_{\lambda,a}$  the function  given by
\begin{equation}\label{DEF-ca}
	c_{\lambda, a}(t) := c(\lambda)(t) + a'(t), \; 
	\lambda \in (0, \sigma(K)),\;t\geq 0.
\end{equation}
Obviously, it follows from Definition \ref{DEF-least-dj2} that $ \lfloor c_{\lambda, a}(\cdot) \rfloor = \lfloor c(\lambda)(\cdot) \rfloor$ for each $\lambda \in (0, \sigma(K))$. 
Next note that 
$$ \lfloor c(\lambda)(\cdot) \rfloor = \lambda^{-1} L(\lambda) + \lambda^{-1} \lfloor \mu \rfloor.$$
Now we state some properties of $\lfloor c(\lambda)(\cdot) \rfloor$ in the following proposition. 
\begin{proposition} \label{PROP-c-dj2}
	Let Assumption \ref{ASS1-dj2} and \ref{ASS2-dj2} be satisfied. Then the following properties hold: 
	\begin{itemize}
		\item[(i)] The map $\lambda \mapsto \left \lfloor c(\lambda)(\cdot) \right \rfloor $ from $(0, \sigma(K))$ to $\R$ is of class $C^1$ from $(0, \sigma(K)) $  into $\R$.
		\item[(ii)] Set
		$ \displaystyle c^*_r:= \inf_{\lambda \in (0, \sigma (K))} \lfloor c(\lambda)(\cdot) \rfloor.$
		There exists $\lambda_r^*\in (0, \sigma(K)]$ such that 
		$$\lim_{\lambda\to (\lambda_r^*)^-}\lfloor c(\lambda) (\cdot) \rfloor= c_r^*.$$
		
		Moreover, one has $c_r^*>0$ and the map        	
		$\lambda \mapsto \lfloor c(\lambda)(\cdot) \rfloor $ is decreasing on $(0, \lambda_r^*)$.
		
		\item[(iii)] Assume that $\lambda_r^* < \sigma(K)$. One has 
		\begin{equation}\label{eq-c*-dj2}
			c^*_r =\int_{\R} K(y) e^{\lambda_r^* y}y \d y.
		\end{equation}
	\end{itemize}
\end{proposition}

The above Proposition \ref{PROP-c-dj2} has been mostly proved in \cite{ducrot2021generalized} (see Proposition 2.8 in \cite{ducrot2021generalized}) with a more general kernel which depends on $t$.

Here we only explain that $ c_r^* > 0$. To see this, note that for $\lambda\in (0,\sigma(K) )$ one has
$$
\lambda c(\lambda)(t)=\int_{\R} K(y) e^{\lambda y}\d y + \mu(t)- \overline{K},\;\forall t\geq 0.
$$
Next due to Assumption \ref{ASS2-dj2} $(f4)$ and Lemma \ref{LEM-least}, there exists some function $a\in W^{1,\infty}(0,\infty)$ such that $\mu(t)-\overline{K}+a'(t)\geq 0$ for all $t\geq 0$. This yields for all $\lambda\in (0,\sigma(K) )$ and $t\geq 0$,
$$
\lambda c(\lambda)(t)+a'(t) = \int_{\R} K(y) e^{\lambda y}\d y+\mu(t)-\overline{K} +a'(t) \geq \int_{\R} K(y) e^{\lambda y}\d y>0,
$$
that rewrites $c_r^*>0$ since $\lfloor a'\rfloor=0$. The result follows.

\begin{remark}
	Let us point out that the assumption $\lambda_r^* < \sigma (K)$ needed for $(iii)$ to hold is satisfied for instance if we have 
	\begin{equation}\label{po}
		\limsup_{\lambda\to \sigma(K)^{-}} \frac{ L(\lambda) }{\lambda} = +\infty.
	\end{equation}
	Indeed, one can observe that 
	$$
	\lfloor c(\lambda)(\cdot)\rfloor \sim \frac{\lfloor \mu\rfloor}{\lambda}\to +\infty\text{  as $\lambda\to 0^+$}.
	$$
	In addition, if \eqref{po} holds, then the decreasing property of the map $\lambda\mapsto \lfloor c(\lambda) (\cdot) \rfloor$ on $(0, \lambda_r^*)$ as stated in Proposition \ref{PROP-c-dj2} $(ii)$ ensures that $\lambda_r^* < \sigma(K)$.
\end{remark}

To state our spreading result, we impose in the following that the condition discussed in the previous remark is satisfied, that means $\lambda_r^*$ is different from the convergence abscissa.
\begin{assumption}\label{ASS3}
	In addition to Assumption \ref{ASS1-dj2}, we assume that $\lambda_r^* < \sigma(K)$.
\end{assumption}

Using the above properties for the speed function $c(\lambda)(\cdot)$ and its least mean value,  we are now able to state  our main results.
% Propagation set, rightmost points.
\begin{theorem}[Upper bounds] \label{THEO_OUT_right-dj2}
	Let Assumption \ref{ASS1-dj2}, \ref{ASS2-dj2} and \ref{ASS3} be satisfied. Let $u=u(t,x)$ denote the solution of \eqref{Pb-dj2} equipped with a continuous initial data $u_0$, with $ 0\leq u_0(\cdot)\leq 1$ and  $u_0(\cdot) \not \equiv 0$. \\
	Then the following upper estimates for the propagation set hold: if $u_0(x)= O(e^{-\lambda x})$ as $x\to\infty$ for some $\lambda>0$, then one has
	$$
	\lim_{t\to\infty} \sup_{x\geq \int_{0}^{t} c^{+}(\lambda)(s)\d s + \eta t }u(t,x)=0,\;\forall \eta >0,
	$$
	where the function $c^+(\lambda)(\cdot)$ is defined by 
	\begin{equation*}
		c^+(\lambda)(\cdot):=\begin{cases} c(\lambda_r^*)(\cdot) &\text{ if $\lambda\geq \lambda_r^*$},\\
			c(\lambda)(\cdot) &\text{ if $\lambda\in (0,\lambda_r^*)$}.
		\end{cases}
	\end{equation*}
	
\end{theorem}

For the lower estimates of the propagation set, we first state our result for a specific function $f=f(t,u)$ of the form
$f(t,u)=\mu(t)(1-u).$
In other words, we are considering the following non-autonomous logistic equation
\begin{equation}\label{sp}
	\partial_{t} u(t,x) = \int_\R K(y) \left[u(t,x-y) -u(t, x)\right]\d y + \mu(t) u(t,x)\left( 1 - u(t,x) \right).
\end{equation}
To enter the framework of Assumption \ref{ASS2-dj2}, we assume that the function $\mu$ satisfies following conditions:
\begin{equation}\label{ASS-mu}
	\begin{split}
		&	t\mapsto \mu(t) \text{ is uniformly continuous and bounded with } \inf_{t\geq 0} \mu(t)>0, \\
		&\text{ and the least mean  of } \mu(\cdot) \text{ satisfies } \lfloor \mu \rfloor >\overline{K}. 
	\end{split}
\end{equation}
For this problem, our lower estimate of propagation set reads as follows.

\begin{theorem}[Lower bounds]\label{THEO_IN_right-dj2}
	Let Assumption \ref{ASS1-dj2}, \ref{ASS3} be satisfied and assume furthermore that $\mu$ satisfies \eqref{ASS-mu}. Let $u=u(t,x)$ denote the solution of \eqref{sp} equipped with a continuous initial data $u_0$, with $ 0\leq u_0(\cdot)\leq 1$ and  $u_0 (\cdot)\not \equiv 0$. 
	Then the following propagation occurs:
	\begin{enumerate}
		\item[(i)] {\bf (Fast exponential decay case)} If $u_0(x)=O(e^{-\lambda x})$ as $x\to\infty$ for some $\lambda\geq \lambda_r^*$, then one has
		$$
		\lim_{t\to\infty} \sup_{x\in [0, ct]}\left|1-u(t,x)\right|=0,\;\forall c\in (0,c_r^*);
		$$
		\item[(ii)] {\bf (Slow exponential decay case)} If $\displaystyle \liminf_{x\to\infty} e^{\lambda x}u_0(x) > 0 $ for some $\lambda\in (0,\lambda_r^*)$, then it holds that
		$$
		\lim_{t\to\infty} \sup_{x\in [0, ct]}\left|1-u(t,x)\right|=0,\;\forall c\in \left(0,\lfloor c(\lambda)\rfloor \right).
		$$
	\end{enumerate}
\end{theorem}

Next as a consequence of the comparison principle, one obtains the following lower estimates of the propagation set to the right-hand side for more general nonlinearity satisfying  Assumption \ref{ASS2-dj2}.
\begin{corollary}[Inner propagation]\label{CORO_IN_right}
	Let Assumption \ref{ASS1-dj2}, \ref{ASS2-dj2} and \ref{ASS3} be satisfied. Let $u=u(t,x)$ denote the solution of \eqref{Pb-dj2} supplemented with a continuous initial data $u_0$, with $ 0\leq u_0(\cdot)\leq 1$ and  $u_0 (\cdot)\not \equiv 0$. 
	Then the following propagation result holds true:
	\begin{enumerate}
		\item[(i)] {\bf (Fast exponential decay case)} If $u_0(x)=O(e^{-\lambda x})$ as $x\to\infty$ for some $\lambda\geq \lambda_r^*$, then one has
		$$
		{\color{red}\liminf_{t\to\infty} }\inf_{x\in [0, ct]} u(t,x)>0,\;\forall c\in (0,c_r^*);
		$$
		\item[(ii)]  {\bf (Slow exponential decay case)} If $\displaystyle \liminf_{x\to\infty} e^{\lambda x}u_0(x) > 0 $ for some $\lambda\in (0,\lambda_r^*)$, then one has
		$$
		{\color{red}\liminf_{t\to\infty}  } \inf_{x\in [0, ct]} u(t,x)>0,\;\forall c\in \left(0,\lfloor c(\lambda)\rfloor \right).
		$$
	\end{enumerate}
\end{corollary}

%
%\begin{remark}
%	If $ \displaystyle \lim_{t\to \infty} \frac{1}{t} \int_{0}^{t} c^{+} (\lambda) (s) \d s = \lfloor c^{+}(\lambda) \rfloor$, then Theorem \ref{THEO_OUT_right-dj2} and Corollary \ref{CORO_IN_right} provide the exact spreading speed $\lfloor c^{+}(\lambda) \rfloor$. This condition holds for instance if $\mu(\cdot)$ has a mean value. \\
%
%	
%\end{remark}

{\color{red}

\begin{remark}
When the coefficients are periodic functions with period $T$, from \cite{jin2009spatial} one can note that $\frac{1}{T} \int_{0}^{T} c^{+} (\lambda)(s)\d s$ is the exact spreading speed for \eqref{Pb-dj2}.  In the periodic situation, our results are also sharp, in the sense that  
$$  \lim_{t\to \infty} \frac{1}{t} \int_{0}^{t} c^{+} (\lambda) (s) \d s = \lfloor c^{+}(\lambda) \rfloor = \frac{1}{T} \int_{0}^{T} c^{+} (\lambda)(s)\d s.$$
The two quantities  $ \lim_{t\to \infty} \frac{1}{t} \int_{0}^{t} c^{+} (\lambda) (s) \d s $ and  $\lfloor c^{+}(\lambda) \rfloor $  also coincide when $c^+(\lambda) (t) $ is a time almost periodic  function.  Therefore our results provide the exact spreading speed for nonlocal  KPP equations in a time almost periodic environment.

In  more general heterogeneous environment,  for instance non-recurrent environment,  one may have  $\displaystyle\lfloor c^{+}(\lambda) \rfloor <\liminf\limits_{t\to \infty} \frac{1}{t} \int_{0}^{t} c^{+}(\lambda) (s) \d s$,
 see Example 1 in \cite{nadin2012propagation}. 
 Our results provide the upper and lower estimates of the propagation set. For $ \beta \in \left( \lfloor c^{+}(\lambda) \rfloor , \liminf\limits_{t\to \infty} \frac{1}{t} \int_{0}^{t} c^{+}(\lambda) (s) \d s \right)$, the behaviour of $u(t,\beta t)$ for $t\gg 1$ is unknown.  This open problem is similar to the non-autonomous Fisher-KPP equation with local diffusion \cite{nadin2012propagation}.

% Our results is not optimal. This may be either caused by our  choice of ways of averaging, or the structure of time heterogeneities leading no exact speed.  For the second situation, we refer the reader to  the examples constructed in \cite[Section 13]{berestycki2019asymptotic}, which considers the spreading properties in reaction-diffusion equations. From such examples with general time heterogeneities, one can see that the level set of $u(t,\cdot)$ can oscillate between two speeds instead of moving with a given speed.  This is a quite interesting phenomena. \\
\end{remark}
}

In the above result we only consider the propagation to the right-hand side of the real line and obtain a propagation result on some interval of the form $[0,ct]$ for suitable speed $c$ and for $t\gg 1$.
Note that the kernel is not assumed to be even, so that the propagation behaviours on the right and the left-hand sides can be different. For instance, different spreading speeds may arise at right and left-hand sides when the kernel is thin-tailed on both sides.
To study the propagation behaviour of the left-hand side, it is sufficient to change $x$ to $-x$ in the above results.

%Note that if we add the new assumption that there exists some constant $\lambda_0<0$ such that $\int_\R K(y) e^{\lambda_0 y} \d y < \infty $. Then one can define $\lambda_l^*$ similar as $\lambda_r^*$, through transforming $x$ to $-x$ one can obtain that the propagation results to the left similar as Theorem \ref{THEO_OUT_right-dj2} and \ref{THEO_IN_right-dj2}. So that we can know the spreading results for the combined initial data, such as $u_0(x)= O(e^{-\lambda x})$ as $x\to \infty$ for some $\lambda >\lambda_r^*$, and $\displaystyle \liminf_{x\to -\infty} e^{\gamma x} u_0(x) > 0 $ for some $\gamma \in (\lambda_l^*, 0) $. Due to the limit of paper length, we omit these details.
%\end{remark}

The results stated in this section and more precisely the lower bounds for the propagation follows from the derivation of suitable regularity estimates for the solution. Here we show that the solutions of \eqref{sp} with suitable initial data are uniformly continuous. Next Theorem \ref{THEO_IN_right-dj2} follows from the application of a general persistence lemma (see Lemma \ref{LEM-key-dj2}) for uniformly continuous solutions. This key lemma roughly ensures that if there is a uniformly continuous solution $u=u(t,x)$ admitting a propagating path $t\mapsto X(t)$, then $[0,kX(t)]$ with any $k\in (0,1)$ is a propagating interval, that is $u$ stays uniformly far from $0$ on this interval, in the large time.  The idea of the proof of this lemma comes from the uniform persistence theory for dynamical systems for which we refer  the reader to \cite{hale1989persistence, magal2005global, smith2011dynamical, zhao2003dynamical} and references cited therein.

This paper is organized as follows. In Section 2, we recall comparison principles and derive our general key persistence Lemma. 
Section 3 is devoted to the derivation of some regularity estimates for the solutions of \eqref{sp} with suitable  initial data. With all these materials, we conclude the proofs of theorems and the corollary.

\section{Preliminary and Key Lemma}

This section is devoted to the statement of the comparison principle and a key lemma that will be used to prove the inner propagation theorem, namely Theorem \ref{THEO_IN_right-dj2}.

\subsection{Comparison principle and strong maximum principle}

We start this section by recalling the following more general comparison principle.

\begin{proposition}(See \cite[Proposition 3.1]{ducrot2021generalized})[Comparison principle]\label{PROP-comparison-dj2}
	Let $t_0\in\R$ and $T>0$ be given. Let $K:\R\to [0,\infty)$ be an integrable kernel and let $F=F(t,u)$ be a function defined in $[t_0,t_0+T]\times [0,1]$ which is Lipschitz continuous with respect to $u\in [0,1]$, uniformly with respect to $t$.
	Let $\underline u$ and $\overline u$ be two uniformly continuous functions defined from $[t_0,t_0+T]\times\R$ into the interval $[0,1]$ such that for each $x\in \R $, the maps $\underline u (\cdot, x)$ and $\overline u (\cdot, x)$ both belong to $W^{1,1} (t_0,t_0+T)$,
	satisfying $\underline u(t_0,\cdot)\leq \overline u(t_0,\cdot)$, and for all $x\in\R$ and for almost every $t\in (t_0,t_0+T)$,
	\begin{equation*}
		\begin{split}
			&\partial_t \overline{u}(t,x) \geq  \int_{\R} K(y) \[\overline{u} (t,x-y) -\overline{u}(t,x)\]\d y +F(t,\overline{u}(t,x)),\\
			&\partial_t \underline{u}(t,x) \leq  \int_{\R} K(y) \[\underline{u} (t,x-y) -\underline{u}(t,x)\]\d y +F(t,\underline{u}(t,x)).
		\end{split}
	\end{equation*}
	Then $\underline u\leq \overline u$ on $[t_0,t_0+T]\times\R$.
\end{proposition}

We  also need some comparison principle on moving domain as follows (this can be proved similarly as Lemma 5.4 in \cite{abi2021asymptotic} and Lemma 4.7 in \cite{zhang2020propagation}).
\begin{proposition}
	Assume that $K: R\to [0,\infty)$ is integrable.  Let $t_0>0$ and $T>0$ be given. Let $b(t,x)$ be a uniformly bounded function from $[t_0, t_0 +T] \times \R \to \R$. Assume that $u(t,x)$ is uniformly continuous defined from $[t_0, t_0 +T ] \times \R$ into the interval $[0,1]$ such that for each $x\in \R$, $u(\cdot, x) \in W^{1,1} (t_0, t_0 +T) $. Assume that $X$ and $Y$ are continuous functions on $[t_0, t_0 +T]$ with $X< Y$. If $u$ satisfies
	\begin{equation*}
		\begin{cases}
			\partial_t u \geq  \int_{\R} K(y) \[u(t,x-y) -u(t,x)\]\d y +b(t,x) u, &\!\forall t\in[t_0, t_0+T], x\in (X(t), Y(t)), \\
			u(t,x)\geq 0, &\!\forall  t\in (t_0, t_0+T], x \in \R \setminus (X(t),Y(t) ), \\
			u(t_0,x) \geq 0, &\! \forall   x\in (X(t_0), Y(t_0)).
		\end{cases}
	\end{equation*}
	Then 
	$$ u(t,x)\geq 0  \text{ for all } t\in [t_0, t_0 +T], x\in [X(t), Y(t)]. $$
\end{proposition}

We continue this section by the following strong maximum principle.	We refer the reader to \cite{kao2010random} for the proof of following proposition.
\begin{proposition}[Strong maximum principle]\label{PROP-Max}
	Let Assumption \ref{ASS1-dj2}, \ref{ASS2-dj2} be satisfied. Let $u=u(t,x)$ be the solution of \eqref{Pb-dj2} supplemented with some continuous initial data $u_0$, such that $0 \leq u_0 \leq 1$ and $u_0 \not \equiv 0$. Then $u(t,x)>0$ for all $t >0, x\in \R$. 
\end{proposition}

\subsection{Key lemma}

In this section, we derive an important lemma that will be used in the next section to prove our main inner propagation result, namely Theorem \ref{THEO_IN_right-dj2}.
In this section we only let Assumption \ref{ASS1-dj2} $(i)$, $(iii)$ and Assumption \ref{ASS2-dj2} be satisfied.

\begin{definition}[Limit orbits set]\label{def-omega-dj2}
	Let $u=u(t,x)$ be a uniformly continuous function on $[0,\infty)\times\R$ into $[0,1]$, solution of \eqref{Pb-dj2}.
	We define $\omega(u)$, {\bf the set of the limit orbits}, as the set of the bounded and uniformly continuous functions $\tilde u:\R^2\to \R$ where exist sequences $(x_n)_n\subset \R$ and  $(t_n)_n \subset [0, \infty)$ such that $t_n\to\infty$ as $n\to\infty$ and
	\begin{equation*}
		\tilde u(t,x)=\lim_{n\to\infty} u(t+t_n,x+x_n),
	\end{equation*}
	uniformly for $(t,x)$ in bounded sets of $\R^2$. 
\end{definition}

Let us observe that since $u$ is assumed to be bounded and uniformly continuous on $[0,\infty)\times\R$, Arzel\`a-Ascoli theorem ensures that $\omega(u)$ is not empty. Indeed, for each sequence $(t_n)_n \subset[0,\infty)$ with $t_n\to\infty$ and $(x_n)\subset \R$ the sequence of functions 
$(t,x)\mapsto u(t+t_n,x+x_n)$ is equi-continuous and thus has a converging subsequence with respect to the local uniform topology.  
In addition, it is a compact set with respect to the compact open topology, that is with respect to the local uniform topology.

Before going to our key lemma, we claim that the set $\omega(u)$ enjoys the following property:
\begin{claim}\label{claim_dichotomy-dj2}
		Let $u= u(t,x)$ be a uniformly continuous solution of \eqref{Pb-dj2}.
	Let $\tilde u\in \omega(u)$ be given. Then one has:
	\begin{equation*}
		\text{Either $\tilde u(t,x)>0$ for all $(t,x)\in\R^2$ or $\tilde u(t,x)\equiv 0$ on $\R^2$}.
	\end{equation*}
\end{claim}

\begin{proof}
	Note that due to Assumption \ref{ASS2-dj2} (see Remark \ref{remark-f-dj2}), the function $u$ satisfies the following differential inequality for all $t\geq 0$ and $x\in\R$
	\begin{equation*}
		\partial_t u(t,x)\geq K\ast u(t,\cdot)(x)-\overline{K}u(t,x)+u(t,x)(\mu(t)-Cu(t,x)).
	\end{equation*}
	Since the function $\mu(\cdot)$ is bounded, for each $\tilde u\in \omega(u)$, there exists $\tilde{\mu}= \tilde{\mu}(t)\in L^\infty(\R)$, a weak star limit of some shifted function $\mu(t_n+\cdot)$, for some suitable time sequence $(t_n)$, such that $\tilde u$ satisfies 
	\begin{equation*}
		\begin{split}
			\partial_t \tilde{u}(t,x)
			&\geq K\ast \tilde{u}(t,\cdot)(x) -  \overline{K} \tilde{u} (t,x) +\tilde{u}(t,x)(\tilde{\mu}(t) -C\tilde{u}(t,x))  \\
			& \geq K\ast \tilde{u}(t,\cdot)(x)+ \left(  -\overline{K} +\inf_{t\in\R} \tilde{\mu}(t) - C \right) \tilde{u} (t,x) ,\;\forall (t,x)\in\R^2.
		\end{split}
	\end{equation*}
	Herein $\partial_{t} \tilde{u}$ is a weak star limit of $\partial_{t} u(\cdot+t_n, \cdot+x_n)$ for some suitable sub-sequence of $(x_n)_n$ and $(t_n)_n$. This is due to $\partial_{t} u \in L^\infty([0, \infty)\times \R)$. 
	
	Next the claim follows from the same arguments as for the proof of the strong maximum principle, see \cite{kao2010random}.
	
\end{proof}

Using the above definition and its properties we are now able to state and prove the following key lemma.

\begin{lemma}\label{LEM-key-dj2}	
	Let $u=u(t,x):[0,\infty)\times\R\to [0,1]$ be a uniformly continuous solution of \eqref{Pb-dj2}. 
	Let $t \mapsto X(t)$ from $[0, \infty)$ to $[0, \infty)$ be a given continuous function.
	Let  the following set of hypothesis be satisfied:
	\begin{itemize}
		\item[(H1)] Assume that 
		$\liminf\limits_{t\to \infty} u(t,0) >0;$ 
		\item[(H2)] There exists some constant $\tilde{\varepsilon}_0>0$  such that for all $\tilde u\in \omega(u)\setminus\{0\}$, one has 
		\begin{equation*}
			\liminf\limits_{t\to \infty} \tilde{u}(t,0) >\tilde{\varepsilon}_0; 
		\end{equation*}
		\item[(H3)] The map $t\mapsto X(t)$ is a propagating path for $u$, in the sense that
		\begin{equation*}
			\liminf\limits_{t\to \infty} u(t, X(t)) >0.
		\end{equation*} 
	\end{itemize}
	Then for any $k\in (0, 1)$, one has 
	$$\liminf\limits_{t\to \infty} \inf_{ 0 \leq x\leq k X(t)} u(t,x)>0.$$
\end{lemma}

\begin{remark}
	The above result holds without assuming that the convolution kernel is exponential bounded.
	We expect this key lemma may also be useful to study the spatial propagation for Fisher-KPP equation with fat-tailed dispersion kernel, {\color{red}where the solution may accelerate}, see \cite{cabre2013influence, finkelshtein2019accelerated, garnier2011accelerating}.
\end{remark} 

To prove the above lemma, we make use of ideas coming from uniform persistence theory, see\cite{hale1989persistence, magal2005global, smith2011dynamical}.  This is somehow close to those developed in \cite{ducrot2021asymptotic, ducrot2019spreading}.

\begin{proof}
	To prove the lemma we argue by contradiction by assuming that there exists $k \in (0, 1)$,  a sequence $(t_n)_n \subset[0,\infty)$ with $t_n \to \infty$ and a sequence $(k_n)$ with $ 0 \leq k_n \leq k$ such that 
	\begin{equation}\label{hyp-contra}		
		u(t_n, k_n X(t_n)) \to 0 \;\text{ as } \; n\to \infty.
	\end{equation}
	
	First we claim that one has
	\begin{equation}\label{claim}
		\lim_{n\to\infty} k_n X(t_n)=\infty.
	\end{equation}		
	To prove this claim we argue by contradiction by assuming that $\{ k_n X(t_n)\}$ has a bounded subsequence. Hence there exists $x_\infty\in \R$ such that possibly along a subsequence still denoted with the index $n$, {\color{red}one has} $k_nX(t_n) \to x_\infty$ as $n\to \infty$. 
	
	Now let us consider the sequence of functions $u_n(t,x):= u(t+t_n, x)$. Since $u=u(t,x)$ is uniformly continuous, possibly up to a sub-sequence still denoted with the same index $n$, there exists $u_\infty\in \omega(u)$ such that
	$$
	u_n(t,x) \to  u_\infty(t,x) \text{ locally uniformly for $(t,x) \in \R^2$.}
	$$
	Next since $k_nX(t_n)\to x_\infty$,  then \eqref{hyp-contra} ensures that	
	$$
	u_\infty(0, x_\infty)= \lim\limits_{n\to \infty} u(t_n,k_n X(t_n)) =0.
	$$
	Since $u_\infty\in \omega(u)$, Claim \ref{claim_dichotomy-dj2}  ensures that $u_\infty(t,x)\equiv 0$.
	On the other hand, $(H1)$ ensures that for all $t\in \R$, one has
	$$
	u_\infty(t,0)\geq \liminf_{t\to\infty} u(t,0)>0,
	$$
	a contradiction, so that \eqref{claim} holds.

	Now due to \eqref{claim}, there exists $N$ such that
	$$
	X(0)<k_nX(t_n),\;\forall n\geq N.
	$$
	{\color{red} Hence due to $k_n<1$ we have
	$$
	X(0) < k_n X(t_n) < X(t_n), \;\forall n\geq N. 
	$$ 
	And since $t\mapsto X(t)$ is continuous, then for each $n\geq N$ there exists $t_n'\in ( 0, t_n )$ such that
	$$
	X(t_n') = k_n X(t_n),\;\forall n\geq N.
	$$
	Since $k_n X(t_n) \to \infty$ as $n\to \infty$ and $t\mapsto X(t)$ is continuous, then $t_n'\to \infty$ as $n\to \infty$.
	}
	
	From the above definition of $t_n'$, one has	
	
	$$u(t_n', k_n X(t_n))= u(t_n', X(t_n')),\;\forall n\geq N.
	$$	
	So that $(H3)$ ensures that for all $n$ large enough,  there exists $\varepsilon>0$ such that 
	$$u(t_n', k_n X(t_n))= u(t_n', X(t_n')) \geq \varepsilon. $$
	Recall that  Assumption $(H2)$. 
	Now  for all $n$ large enough, we define 
	$$t_n'' :=\inf \left \{ t \leq t_n; \;  \forall s\in (t,t_n), \; u(s, k_n X(t_n)) \leq \frac{ \min\{\tilde{\varepsilon}_0, \varepsilon \} }{2} \right\} \in (t_n', t_n).$$	
	Since $u(t_n,  k_n X(t_n)) \to 0$ as $n\to \infty$, then one may assume that, for all $n$ large enough one has 
	\begin{equation*}
		\begin{cases}
			u(t_n'', k_n X( t_n)) = \frac{\min\{\tilde{\varepsilon}_0, \varepsilon\}}{2}, \\
			u(t,  k_n X( t_n)) \leq \frac{\min\{\tilde{\varepsilon}_0, \varepsilon\}}{2}, \; \forall t\in [t_n'', t_n], \\
			u(t_n, k_n X(t_n)) \leq \frac{1}{n}. 
		\end{cases} 
	\end{equation*}
	Next we claim that $t_n- t_n'' \to \infty$ as $n\to \infty$. Indeed, if (a subsequence of) $t_n -t_n''$ converges to $\sigma \in \R$, define the sequence of functions $\tilde{u}_n(t,x):= u(t+t_n'', x+ k_n X(t_n))$, that converges, possibly along a subsequence, locally uniformly to some function $\tilde {u}_\infty=\tilde {u}_\infty (t,x)\in \omega(u)$ with
	$$
	\tilde{u}_\infty(0,0)=\frac{\min\{\tilde{\varepsilon}_0, \varepsilon\}}{2}>0,
	$$
	and
	$$
	\tilde{u}_\infty(\sigma, 0) = \lim\limits_{n \to \infty} \tilde{u}_n(t_n -t_n'', 0 ) = \lim\limits_{n \to \infty} u(t_n , k_n X(t_n) ) =0. 
	$$
	Since $\tilde u_\infty \in \omega(u)$, then the above two values of $\tilde u_\infty$ contradict the dichotomy stated in Claim \ref{claim_dichotomy-dj2} and this proves that $ t_n- t_n'' \to \infty$ as $n\to \infty$. 
	
	As a consequence one obtains that the function $\tilde u_\infty \in \omega(u)$ satisfies
	$$
	\tilde{u}_\infty(0,0)=\frac{\min\{\tilde{\varepsilon}_0, \varepsilon\}}{2}>0,
	$$
	together with
	\begin{equation} \label{contradict}
		\tilde{u}_\infty(t,0) \leq \frac{\min\{\tilde{\varepsilon}_0, \varepsilon\}}{2}, \; \forall t\geq 0.
	\end{equation}
	Due to Claim \ref{claim_dichotomy-dj2}, the above equality yields $\tilde{u}_\infty\in\omega(u)\setminus\{0\}$ and \eqref{contradict} contradicts $(H2)$. The proof is completed.
\end{proof}

\section{Proof of spreading properties}
In this section, we shall make use of the key lemma (see Lemma \ref{LEM-key-dj2}) to prove Theorem \ref{THEO_IN_right-dj2}. To do this, we first derive some important regularity properties of the solutions of the Logistic equation \eqref{sp} associated with suitable initial data.
Next we prove Theorem \ref{THEO_OUT_right-dj2} by constructing suitable exponentially decaying super-solutions for \eqref{Pb-dj2}. Finally we turn to the proof of Theorem \ref{THEO_IN_right-dj2}. As already mentioned we crucially make use of Lemma \ref{LEM-key-dj2} and construct a suitable propagating path $t\mapsto X(t)$, that depends on the decay rate of the initial data $u_0=u_0(x)$ for $x\gg 1$. As a corollary, we conclude the propagation results for \eqref{Pb-dj2}.

\subsection{Uniform continuity estimate} \label{sectionkey}
This subsection is devoted to giving some regularity estimates for the solutions of the following Logistic equation (recalling \eqref{sp}) when endowed with suitable initial data, 
$$ 
\partial_{t} u(t,x) = \int_\R K(y) u(t,x-y) \d y - \overline{K} u(t,x) + \mu(t) u(t,x)\left( 1 - u(t,x) \right).
$$
Here we focus on two types of initial data, that will be used to prove  Theorem \ref{THEO_IN_right-dj2}: initial data with a compact support and initial data with support on a right semi-infinite interval and with some prescribed exponential decay on this right-hand side (that is for $x\gg 1$).

Our first lemma is concerned with the compactly supported case.
\begin{lemma}\label{LEM-uc}
	Let Assumption \ref{ASS1-dj2} and \eqref{ASS-mu} be satisfied. Let $u=u(t,x)$ be the solution of \eqref{sp}  equipped with the initial data  $v_0=v_0(x)$, where $v_0$ is Lipschitz continuous in $\R$,  and $0<v_0(x)<1$ for all $x \in (0, A)$, for some constant $A>0$ while $v_0=0$  outside of $(0, A)$. Then, the function $(t,x)\mapsto u(t,x)$ is uniformly continuous on $[0,\infty)\times \R$.
\end{lemma}

\begin{proof}
	Firstly, since $ 0 \leq u \leq 1$, then one has
	\begin{equation}\label{lip1}	
		\|\partial_t u\|_{L^\infty(\R^+\times \R)} \leq M:=2 \overline{K} +\|\mu \|_\infty.
	\end{equation}
	As a consequence, the map $(t,x)\mapsto u(t,x)$ is  Lipchitz continuous for the variable $t\in [0,\infty)$, uniformly with respect to $x\in\R$, that is
	\begin{equation}\label{lip2}
		|u(t,x)-u(s,x)|\leq M|t-s|,\;\forall (t,s)\in [0,\infty)^2,\;\forall x\in\R.
	\end{equation}

	Next we investigate the regularity with respect to the spatial variable $x\in\R$. To do so we claim that the following holds true:
	\begin{claim}\label{cl}
		For all $h>0$  sufficiently small, there exists $0<\sigma (h)<1$ such that $\sigma(h)\to 1$ as $h\to 0$ and 
		$$
		u(\sqrt{h}, x) \geq \sigma (h) v_0(x-h),\;\forall x\in\R.
		$$
	\end{claim}	
	
	\begin{proof}[Proof of Claim \ref{cl}]
		Let us first observe that since $u(t,.)>0$ for all $t>0$, it is sufficient to look at $x-h\in [0,A]$, that is $h\leq x\leq A+h$.
		
		Next to prove this claim, note that one has for all $h>0$ and $x\in\R$:
		\begin{equation*}
			\begin{split}
				u(\sqrt{h},x) & = v_0(x) + \int_0^{\sqrt{h}} \partial_t u(l,x)\d l\\
				& = v_0(x) + \int_0^{\sqrt{h}}\left\{\int_{\R} K(y) \left[ u(l,x-y) -u(l,x) \right] \d y + \mu(l) u(l,x) \left(1- u(l,x) \right)\right\}\d l
			\end{split}
		\end{equation*}
		Now coupling \eqref{lip2} and  $0\leq u \leq 1$, one gets, for all $h>0$ small enough and uniformly for $x\in\R$
		
		\begin{equation*}
			u(\sqrt{h},x) \geq  v_0(x) + \int_0^{\sqrt{h}}\left\{\int_{\R} K(y) v_0(x-y)\d y -\overline{K} v_0(x) \right \}\d l+o(\sqrt h),
		\end{equation*}
		that is
		\begin{equation*}
			u(\sqrt{h},x)  \geq  v_0(x) \bigg ( 1 - \overline{K} \sqrt{h} \bigg) + \sqrt{h} \bigg(  \int_{\R} K(y) v_0(x-y)\d y + o(1) \bigg).
		\end{equation*}
		Now observing  Assumption \ref{ASS1-dj2} (see $(i)$ and $(iii)$), there exists $\varepsilon>0$ such that
		$$
		\min_{x\in [0,A]}\int_{\R} K(y) v_0(x-y)\d y\geq 2\varepsilon,
		$$	
		so that	for $h>0$ small enough one has
		$$
		\min_{x\in [h,A+h]}\int_{\R} K(y) v_0(x-y)\d y\geq \varepsilon,
		$$		
		Now to prove the claim, it is sufficiently to reach, for all $h>0$ small enough and $ x\in [h,A+h]$,
		\begin{equation} \label{eq-v0h}
			v_0(x) \bigg ( 1 - \overline{K} \sqrt{h} \bigg) + \sqrt{h} \left( o(1) + \varepsilon\right) \geq \sigma (h) v_0(x-h).
		\end{equation}
		Now set $\sigma (h) = 1 - 2\overline{K} \sqrt{h}$ and let us show that Claim \ref{cl} follows.	
		
		Since $v_0$ is Lipschitz continuous, then there exists some constant $L>0$ such that 
		$$
		|v_0(x)-v_0(x-h)|\leq L h, \;\forall x\in\R.
		$$	
		Hence to obtain \eqref{eq-v0h},  it is sufficient to reach for all $x\in [h, A+h]$ and all $h>0$ small enough
		\begin{equation} \label{eq-v0h1}
			\overline K \sqrt{h} v_0(x-h) + \sqrt{h} \left( o(1) + \varepsilon\right) \geq L h \bigg ( 1 - \overline {K} \sqrt{h} \bigg).
		\end{equation}
		Dividing by $\sqrt h$, the above inequality holds whenever
		\begin{equation} \label{eq-v0h2}
			\overline K v_0(x-h) + \left( o(1) + \varepsilon\right) \geq L \sqrt{h}\bigg ( 1 - \overline {K} \sqrt{h} \bigg),
		\end{equation}
		which holds true for all $h>0$ small enough. So the claim is proved. 
		
	\end{proof}

	Now we come back to the proof of Lemma \ref{LEM-uc}. For each $h>0$ small enough, let us  introduce the following function
	\begin{equation}\label{eq-b}
		b_h (t)= b_h(0) \exp\left\{\int_{0}^{t}\left[\mu(s+\sqrt{h}) - \mu(s) \right]\d s \right\}, \; \text{ for all } t\geq 0,
	\end{equation}
	where $b_h(0)$ is some constant depending on $h$ and that satisfies the following three conditions:\\
	$$ 0<b_h(0) \leq \sigma(h)< 1, $$
	$b_h(0) \to 1$ as $h\to 0$  and for all $h>0$ small enough
	$$
	b_h(0)\leq \inf_{t\geq 0} \frac{\mu(t)}{\mu(t+ \sqrt{h})} \exp\left\{\int_{0}^{t} \left[ \mu(s)- \mu(s+ \sqrt{h}) \right] \d s \right \}. $$
	For the later condition, one can observe that it is feasible since one has
	\begin{equation*}
		\begin{split}
			\left| \int_{0}^{t} \left[ \mu(s+ \sqrt{h}) - \mu(s) \right] \d s \right|
			&= \left|\int_{\sqrt{h}}^{t+\sqrt{h}} \mu (s) \d s - \int_{0}^{t}\mu(s) \d s \right|  \\
			&= \left| \int_{t}^{t+\sqrt{h}} \mu(s) \d s - \int_{0}^{\sqrt{h}} \mu(s) \d s\right| \\
			& \leq 2\|\mu\|_\infty \sqrt{h}.
		\end{split}
	\end{equation*}
	As a consequence, recalling \eqref{ASS-mu}, $\mu(\cdot)$ is uniformly continuous and we end-up with 
	$$
	\frac{\mu(t)}{\mu(t+ \sqrt{h})} \exp\left\{\int_{0}^{t} \left[ \mu(s)- \mu(s+ \sqrt{h}) \right] \d s \right \} \to 1, \text{ as } h \to 0, \text{ uniformly for } t\geq 0.
	$$
	Hence $b_h(0)$ is well defined and $b_h(t) \to 1$ as $h\to 0$ uniformly for $t\geq 0$.
	
	Now, setting $w_h=w_h(t,x)$ the function given by
	$$
	w_h(t,x):= u(t+ \sqrt{h}, x) - b_h(t) u(t,x-h),$$
	one obtains that it  becomes a solution of the following equation
	\begin{equation*}
		\begin{split}
			\partial_t w_h (t,x) & = K\ast w_h(t,x) - \overline{K} w_h(t,x)  \\
			&\quad +\mu(t+\sqrt{h}) \left[w_h(t,x) + b_h(t) u(t,x-h)\right] \left[1- \left( w_h(t,x) + b_h(t) u(t,x-h) \right) \right] \\
			&\quad - \mu(t) b_h(t) u(t,x-h) \left[ 1- u(t,x-h) \right] - b'_h(t) u(t, x-h) \\
			& = K\ast w_h(t,x) - \overline{K} w_h(t,x)  + \mu(t+\sqrt{h}) w_h(t,x) \bigg(1- w_h (t,x) -2b_h(t) u(t,x-h) \bigg) \\
			& \quad + b_h(t) u(t,x-h) \left( \mu(t+ \sqrt{h})- \mu(t)- \frac{b_h'(t)}{b_h(t)} \right)  + b_h(t) u^2(t,x-h) \left( \mu(t) - b_h (t) \mu(t+\sqrt{h}) \right). 
		\end{split}
	\end{equation*}   
	It follows from  the definition of $b_h(t)$ (see \eqref{eq-b} above) that $w_h(t,x)$ satisfies 
	\begin{equation*}
		\partial_t w_h(t,x) \geq K\ast w_h (t,x) -\overline{K} w_h(t,x) + w_h(t,x) \mu(t+\sqrt{h}) \bigg( 1- w_h (t,x) - 2b_h(t) u(t,x-h) \bigg).
	\end{equation*}
	The Claim \ref{cl} together with $b_h(0) <\sigma(h)$ ensure that $w_h(0, \cdot) \geq 0$. Then the comparison principle applies and implies that 
	$w_h(t,x) \geq 0$ for all $t\geq 0, x\in \R$, that rewrites as  $u(t+ \sqrt{h}, x) \geq b_h(t) u(t,x-h)$ for all $t\geq 0, x\in \R$, for $h>0$ small enough. Recalling \eqref{lip2}, for  $ h>0$ sufficiently small, one has for all $t\geq 0$ and $x\in\R$,
	\begin{equation}\label{eq-h1}
		u(t,x-h) - u(t,x) \leq \left( \frac{1}{b_h(t)} - 1 \right) u(t + \sqrt{h}, x) + M \sqrt{h} \leq \left( \frac{1}{b_h(t)} - 1 \right)  + M \sqrt{h}.
	\end{equation} 
	Since for $h>0$ small enough one has 
	$$
	\min_{x\in [-h,A-h]}\int_{\R} K(y) v_0(x-y)\d y\geq \varepsilon,
	$$		
	then one can similarly  prove that for sufficiently small $h>0$, there exists $\sigma(h)= 1- 2\overline{K} \sqrt{h}$  such that 
	$$u( \sqrt{h}, x) \geq \sigma (h) v_0 ( x+h), \; \forall x\in \R.$$
	This rewrites as  
	$$u( \sqrt{h}, x-h) \geq \sigma (h) v_0 ( x), \; \forall  x\in \R. $$
	Then as above one can choose a suitable function $b_h(t)$ and obtain that 
	\begin{equation*}
		u(t+ \sqrt{h}, x-h) \geq  b_h(t) u(t,x), \; \forall t\geq 0, x\in\R. 
	\end{equation*}
	Recalling \eqref{lip2}, for $h>0$ sufficiently small, one obtains for all $t\geq 0$ and $x\in\R$,
	\begin{equation}\label{eq-h2}
		\begin{split}
			u(t,x) -u(t,x-h)  &\leq \left(  \frac{1}{b_h(t)} -1 \right) u(t+ \sqrt{h} , x-h) + M \sqrt{h} \\
			&  \leq \left( \frac{1}{b_h(t)} - 1 \right)  + M \sqrt{h}.
		\end{split}
	\end{equation} 
	Since estimates \eqref{eq-h1} and \eqref{eq-h2} are uniform with respect to the spatial variable $x\in \R$, one also obtains a similar estimates for $u(t,x) -u(t,x+h) $ and $u(t,x+h) -u(t,x)$. 
	From these estimates one has reached that
	$u=u(t,x)$ is uniformly continuous for all $t\geq 0, x\in \R$, which completes the proof of the lemma.  
\end{proof}

In the following we derive regularity estimates for the solutions to \eqref{sp} coming from an initial data with a prescribed exponential decay rate of the right, that for $x\gg 1$. To do this, we show that such solutions to \eqref{sp} decay with the same rate as the initial data, at least in a short time. 

Let us introduce some function spaces. Recalling that $\lambda_r^*$ is defined in Proposition \ref{PROP-c-dj2}, for $\lambda \in ( 0 , \lambda_r^*) $ let us define the space $BC_\lambda(\R)$ by 
$$
BC_\lambda(\R) := \left\{  \phi \in C(\R): \;   \sup_{x\in\R}  e^{\lambda x} |\phi(x)| < \infty  \;   \right\},
$$
equipped with the weighted norm 
$$
\| \phi \|_{BC_\lambda} :=\sup_{x\in\R} e^{\lambda x} |\phi(x)|. 
$$
Recall that $BC_\lambda(\R)$ is a Banach space when endowed with the above norm.

Define also the subset $E$ by
\begin{equation}
	E:= \left\{ \phi \in BC_\lambda(\R): 0\leq \phi \leq 1  \right\},
\end{equation}
and let us observe that it is a closed subset of $BC_\lambda(\R)$.

Using these notations, we turn to the proof of the following lemma.

\begin{lemma}\label{LEM-decay} 
	Let Assumption \ref{ASS1-dj2} and \ref{ASS3} and \eqref{ASS-mu} be satisfied.
	Let $\lambda\in (0,\lambda_r^*)$ and $u_0\in E$ be given. Then the solution of \eqref{sp} with initial data $u_0$, denoted by $u=u(t,x)$, satisfies
	$$
	\lim\limits_{t\to 0^+} \sup_{x\in \R} e^{\lambda x} |u(t,x) - u_0(x)|=0.
	$$
\end{lemma}

\begin{proof}
	Fix $\alpha>\overline{K}+ 2\| \mu \|_\infty $. Let us introduce for each $\phi \in E$ and 
	$t \geq 0$, the operator given by
	$$
	Q_t[\phi](\cdot):= \alpha \phi(\cdot) + \int_\R K(y)\phi(\cdot-y)\d y -\overline{K} \phi (\cdot)+ \mu(t) \phi(\cdot)\left(1- \phi(\cdot) \right). $$
	Note that one has
	$$
	\left \|\int_\R K(y)\phi(\cdot - y)\d y \right \|_{BC_\lambda} = \sup_{x\in \R}  \left|\int_\R K(y) e^{\lambda y} e^{\lambda (x -y) } \phi(x-y) \d y \right| \leq \left|\int_\R K(y) e^{\lambda y} \d y \right| \left\|\phi \right\|_{BC_\lambda}.
	$$
	Let us observe that $\left|\int_\R K(y) e^{\lambda y} \d y \right|< \infty$ due to $ 0<\lambda < \lambda_r^* < \sigma(K) $. Since $ 0\leq \phi \leq 1$ then one has
	$$\| Q_t[\phi](\cdot) \|_{BC_\lambda} \leq \left(\alpha + \left| \int_\R K(y) e^{\lambda y} \d y \right| + \overline{K} + \| \mu \|_\infty \right) \|\phi\|_{BC_\lambda}<\infty. $$
	Thus for each $\phi(\cdot) \in E$, for all $t\geq 0$, $ Q_t [\phi]( \cdot) \in BC_\lambda (\R)$.
	
	Next let us observe that $Q_t[\phi]$ is nondecreasing with respect to $\phi\in E$. Indeed, if for any $ \phi, \psi\in E$ and $\phi(x) \geq \psi(x) $ for all $x\in\R$, 
	then for each given $t\geq 0, x\in\R$ 
	\begin{equation*}
		\begin{split}
			Q_t[\phi](x) -Q_t[\psi](x)  &= \alpha (\phi(x)-\psi(x)) + \int_\R K(y)[\phi(x-y) - \psi(x-y)] \d y - \overline{K} (\phi-\psi)(x)  \\
			&\quad + \mu(t) \phi(x)( 1 -\phi(x)) -\mu(t) \psi(x)( 1 -\psi(x) ) \\
			&\geq \left( \alpha - \overline{K} - 2 \|\mu \|_\infty \right) \left(\phi(x)-\psi(x) \right) \\
			&\geq 0.
		\end{split}
	\end{equation*}
	The last inequality comes from $\alpha>\overline{K}+ 2 \|\mu\|_\infty$. So that for any $t\geq 0$, the map $\phi\mapsto Q_t[\phi]$ is nondecreasing on $E$.

	For each given $u_0 \in E$ and any fixed $h>0$, we define the following space
	$$
	W:= \left\{ t\mapsto  u(t, \cdot) \in C([0,h], BC_\lambda(\R)): \; 0\leq  u  \leq 1, u(0, x)=u_0(x) \right\}.
	$$
	Let us rewrite \eqref{sp} to 
	\begin{equation*}
		\partial_t u(t,x) + \alpha u(t,x) = Q_t[u(t,\cdot)](x), 
	\end{equation*}
	then one has
	\begin{equation*}
		u(t,\cdot)= e^{-\alpha t} u_0(\cdot) + \int_{0}^{t} e^{\alpha (s-t)} Q_s[u(s,\cdot)](\cdot) \d s  =: T[u](t,\cdot).
	\end{equation*}
	
	Next we show that for each $u \in W$, one has $ T[u] \in W$. Let  $u\in W$ be given. Firstly we show that $Q_t[u](\cdot) \in BC_\lambda(\R)$ uniformly for $t\in [0, h]$. Since $t\mapsto u(t, \cdot) \in C([0,h], BC_\lambda(\R) )$, then one has 
	$$\sup_{t\in[0,h]} \|u(t,\cdot)\|_{BC_\lambda} <\infty.$$
	Thus one has 
	$$\sup_{ t\in[0,h] } \| Q_t[u(t,\cdot)](\cdot) \|_{BC_\lambda} \leq \left(\alpha + \left| \int_\R K(y) e^{\lambda y} \d y \right| + \overline{K} + \|\mu\|_\infty \right) \sup_{ t\in[0,h] } \|u(t,\cdot)\|_{BC_\lambda}<\infty. $$
	Moreover, one can observe that  for each $t\in [0, h]$,
	$$ \| T[u](t,\cdot) \|_{ BC_\lambda } \leq \| u_0 \|_{BC_\lambda} + \frac{1}{\alpha} \sup_{ t\in[0,h] } \|Q_t[u(t,\cdot)] \|_{BC_\lambda} < \infty. $$
	That is $T[u](t,\cdot) \in BC_\lambda(\R)$, for each $t\in [0, h]$.
	
	Then we show that $t\mapsto T[u](t,\cdot)$ is continuous. To see this, fix $t_0 \in [0,h]$ and observe that one has 
	\begin{equation*}
		\begin{split}
			\left\| T[u](t,\cdot) - T[u](t_0, \cdot) \right\|_{BC_\lambda}
			&\leq \left|e^{-\alpha t} - e^{-\alpha t_0}\right| \| u_0\|_{BC_\lambda} \\
			& \quad  + \sup_{x\in\R} e^{\lambda x}\left| \int_{0}^{t_0} \left[ e^{\alpha(s-t)} -e^{\alpha(s-t_0)} \right] Q_s[u(s,\cdot)](x) \d s \right|\\
			& \quad + \sup_{x\in\R} e^{\lambda x} \left| \int_{t_0}^{t} e^{\alpha( s-t)} Q_s[u(s,\cdot)](x) \d s  \right| \\
			&\leq  \left|e^{-\alpha t} - e^{-\alpha t_0}\right| \| u_0\|_{BC_\lambda}  \\
			& \quad +  \left|e^{-\alpha t} - e^{-\alpha t_0}\right| \sup_{ s\in[0,h] } \|Q_s[u(s,\cdot)] \|_{BC_\lambda}  \int_{0}^{t_0} e^{\alpha s} \d s \\
			&\quad + \sup_{ s\in[0,h] } \|Q_s[u(s,\cdot)]\|_{BC_\lambda} \left|\frac{1- e^{ \alpha(t_0-t)}}{\alpha} \right|.
		\end{split}
	\end{equation*}  
	So that $t\mapsto T[u](t, \cdot) \in C([0, h], BC_\lambda(\R))$ and $T[u](0,\cdot)= u_0(\cdot)$.
	
	Also,  note that 
	due to for each $t\in [0, h] $,  $Q_t[u(t,\cdot)]$ is nondecreasing with $u(t, \cdot)\in E$,  then we get 
	$$ 0\leq T [u](t,\cdot) \leq  e^{-\alpha t}  + \frac{1}{\alpha} (1- e^{-\alpha t}) \alpha  \leq 1, \;\; \forall t\in [0, h]. $$
	Hence, for each $u \in W$, then $T[u] \in W$.
	
	For each $u, v\in W$ and a given $\gamma>0$ large enough, we introduce a metric on $W$ defined by
	$$ d(u,v):= \sup_{t\in [0, h]} \sup_{x\in \R} e^{\lambda x} |u(t,x) - v(t,x) | e^{-\gamma t}. $$ 
	Note that 
	\begin{equation*}
		\begin{split}
			&d(T [u], T[v] )  = \sup_{t\in[0, h]} \sup_{x\in \R } e^{\lambda x } \left| \int_{0}^{t} e^{\alpha(s-t)} \left( Q[u] (s,x) -Q[v](s,x)\right) \d s \right| e^{-\gamma t}\\
			& \leq  \sup_{t\in[0, h]} \sup_{x\in \R } \left| \int_{0}^{t} e^{(\alpha + \gamma) (s-t)} \left[\alpha+ \int_\R K(y) e^{\lambda y} \d y + \overline{K} + 3 \|\mu \|_\infty  \right] e^{-\gamma s} e^{\lambda x} |u(s,x)-v(s,x)| \d s \right| \\
			&\leq  \left[ \alpha+ \int_\R K(y) e^{\lambda y} \d y + \overline{K} + 3 \| \mu \|_\infty \right] \sup_{t\in[0,h]} \int_{0}^{t} e^{(\alpha+\gamma) (s-t)} \d s \cdot d(u,v) \\
			& \leq \frac{\alpha+ \int_\R K(y) e^{\lambda y} \d y + \overline{K} + 3 \| \mu \|_\infty }{\alpha + \gamma} \cdot d(u, v).
		\end{split}
	\end{equation*}
	So that $T[u] $ is a contraction map on $W$ endowed with the metric $d=d(u,v)$, as long as $\gamma >0$ sufficiently large such that 
	$$\frac{\alpha+ \int_\R K(y) e^{\lambda y} \d y + \overline{K} + 3 \| \mu \|_\infty }{\alpha + \gamma}<1.$$
	Finally since $(W, d)$ is a complete metric space, by Banach fixed point theorem ensures that  $T[u]$ has a unique fixed point in $W$ which is the solution of \eqref{sp} with $u(0, \cdot)= u_0(\cdot)$. Since $t\mapsto u(t,\cdot) \in C([0,h], BC_\lambda(\R))$, then one has obtained
	$$	
	\lim\limits_{t\to 0^+} \sup_{x\in \R} e^{\lambda x} |u(t,x) - u_0(x)|=0,
	$$
	that completes the proof of the lemma.
\end{proof}

\begin{lemma}\label{LEM-Exp}
	Let Assumption \ref{ASS1-dj2} and \ref{ASS3} and \eqref{ASS-mu} be satisfied. Let $u=u(t,x)$ be the solution of \eqref{sp} supplemented with the initial data  $v_0$ satisfying the following properties:\\
	assume $v_0$ is Lipschitz continuous in $\R$, there is $A>0$ large enough, $\alpha >0$, $p\in (0, 1)$ and  $\lambda \in (0, \lambda_r^*)$ such that
	\begin{equation}
		v_0(x)= \begin{cases}
			\text{increasing function}, \; & x\in [0, \alpha], \\
			\beta:= p e^{- \lambda A}, \; & x\in [\alpha, A], \\
			p e^{-\lambda x}, \; & x \in [A, \infty), \\
			0, \; & x\in (-\infty, 0].
		\end{cases}
	\end{equation} 
	Then the function $u=u(t,x)$ is uniformly continuous on $[0,\infty)\times\R$.
\end{lemma}
\begin{proof}
	As in the proof of Lemma \ref{LEM-uc}, $u=u(t,x)$ also satisfies \eqref{lip2}.\\
	{\color{red} Now from the definition of $v_0$, for $h>0$ small enough,  for the  given $\lambda\in (0, \lambda_r^*)$,  one can observe 
	$$ v_0(x) \geq e^{-\lambda h} v_0 (x-h), \; \forall x\in \R.$$
%	Indeed, for $x\leq A $, due to $v_0$ is nondecreasing on this interval, then $v_0(x) \geq v_0(x-h)\geq e^{-\lambda h} v_0 (x-h)$ for $x\leq A$.  While for $x\in x\in [A, A+h]$, one has
%	$$v_0(x) \geq v_0(A+h) = p e^{-\lambda (A+h)}= e^{-\lambda h} v_0 (x-h), \; \forall x\in [A, A+h].$$
%	For $x\geq A + h $, one has
%	$$v_0(x) = p e^{-\lambda x} = e^{-\lambda  h} v_0(x-h), \; \forall x\geq A+h. $$
%	Thus we have obtained 
%	$$
%	v_0(x) \geq e^{-m h} v_0(x-h),\;\;\forall x\in \R.
%	$$
Let us show that the function $v^h(t,x) := e^{-\lambda h} u(t,x-h) $ (with $v^h(0,x) = e^{-\lambda h} v_0(x-h)$) is a sub-solution of \eqref{sp}. To see this, note that $v^h(t,x)$ satisfies
	\begin{equation*}
		\begin{split}
			\partial_t v^h(t,x) &= \int_\R K(y) v^h(t,x-y) \d y -\overline{K} v^h(t,x) + \mu(t) v^h(t,x) \left( 1 - e^{\lambda  h} v^{h}(t,x) \right) \\
			&\leq \int_\R K(y) v^h(t,x-y) \d y -\overline{K} v^h(t,x) + \mu(t) v^h(t,x) \left( 1 - v^{h}(t,x) \right).
		\end{split}
	\end{equation*}
	Hence $v^h (t,x)$ becomes a sub-solution of \eqref{sp}.   

	Since $v^h(0, \cdot) \leq v_0(\cdot)$, the comparison principle implies that 
	$$u(t,x) \geq e^{-\lambda h} u(t,x-h), \; \forall t\geq 0, x\in \R.$$
	Similarly as in \eqref{eq-h1}, one also has, for all $h>0$ sufficiently small,
	\begin{equation} \label{uch1}
		u(t,x-h) - u(t,x) \leq \left( 1 -e^{-\lambda h} \right) u(t,x-h) \leq 1 -e^{-\lambda h}, \forall t\geq 0, x\in\R,
	\end{equation}
	and changing $x$ to $x+h$ yields for all $h >0$ sufficiently small,
	\begin{equation}\label{uch2}
		u(t,x) - u(t,x+h) \leq \left( 1 -e^{-\lambda h} \right) u(t,x) \leq 1 -e^{-\lambda h}, \; \forall t\geq 0, x\in \R.
	\end{equation}
}
	
	Next we show that there exists $0 < \alpha(h) <1$, $\alpha(h) \to 1$ as $h \to 0$ such that for  all $h>0$ small enough
	$$u(\sqrt{h}, x) \geq \alpha(h) v_0(x+h), \; \forall x\in \R. $$
	Since  $v_0(x+h)=0$ for $x\leq -h$, it is sufficiently to consider the above inequality for $x \geq -h$. As in the proof of Lemma \ref{LEM-uc}, 
	note that for all $h>0$ sufficiently small and uniformly for $x\in\R$, one has
	\begin{equation*}
		u(\sqrt{h},x)  \geq  v_0(x) \bigg ( 1 - \overline{K} \sqrt{h} \bigg) + \sqrt{h} \bigg(  \int_{\R} K(y) v_0(x-y)\d y + o(1) \bigg).
	\end{equation*} 
	One may now observe that for all $ 2A \geq x\geq -h$, there exists $\varepsilon>0$ such that 
	$$\int_{\R} K(y) v_0(x-y)\d y \geq \varepsilon >0. $$  
	As in the proof of Claim \ref{cl}, set $\alpha_1 (h)= 1-2\overline{K}\sqrt{h}$. Then one has 
	$$u(\sqrt{h}, x) \geq \alpha_1 (h) v_0(x+h), \;\forall x\leq 2A.$$

	Let us now prove that there exists $0<\alpha_2(h)<1$ and $\alpha_2(h) \to 1,$ as $h\to 0$ such that $u(\sqrt{h}, x) \geq \alpha_2(h) v_0(x+h)$ for $x\geq 2A$. 
	From Lemma \ref{LEM-decay}, one has 
	$$\lim\limits_{h\to 0^+} \sup_{x\geq 2A} e^{\lambda x} |u(\sqrt{h}, x) - p e^{-\lambda x}|=0.$$ 
	Set
	$$
	\gamma(h) := \sup_{x\geq 2A} e^{\lambda x} |u(\sqrt{h}, x) - p e^{-\lambda x}|,
	$$
	and observe that, for $h$ sufficiently small, for all $x\geq 2A$, one has
	\begin{equation*}
		\begin{split}
			\left( 1- \frac{\gamma(h)}{p} \right) v_0(x) =-\gamma(h) e^{-\lambda{x}} + p e^{-\lambda x} & \leq u(\sqrt{h}, x)  \\
			&\leq \gamma(h) e^{-\lambda{x}} + p e^{-\lambda x} \\
			&=  \left( \frac{\gamma(h)}{p} + 1 \right) v_0(x).
		\end{split}
	\end{equation*}
	So that one can set $\alpha_2(h) := 1- \frac{\gamma(h)}{p}$ to obtain $0<\alpha_2(h)<1$, $\alpha_2(h) \to 1$ as $h\to 0$ and
	$$ 
	u(\sqrt{h},x) \geq \alpha_2(h) v_0(x), \; \forall x\geq 2A. 
	$$
	Then since $v_0$ is non-increasing  for $x\geq A$, one has 
	$$
	u(\sqrt{h}, x) \geq \alpha_2(h) v_0(x) \geq \alpha_2 (h) v_0(x+h), \; \forall x\geq 2A. 
	$$
	Now, set $\alpha(h):= \min\{\alpha_1(h), \alpha_2(h)\}$.  We get 
	$$
	u(\sqrt{h}, x) \geq \alpha(h) v_0(x+h), \; \forall x\in\R. 
	$$
	
	As in the proof of Lemma \ref{LEM-uc}, one can also construct a function $\tilde{b}_h(t)\to 1$ as $h\to 0$ uniformly for $t\geq 0$ with $0<\tilde{b}_h(0)<\alpha(h) $ and such that for all $h>0$ small enough one has
	$$ u(t+ \sqrt{h},x) \geq \tilde{b}_h(t) u(t,x+h), \; \forall t\geq 0, x\in \R. $$
	With such a choice, for all $h >0$ small enough, for all $t\geq 0$ and $x\in \R$,  one obtains  that
	\begin{equation}
		u(t,x+h) -u(t,x) \leq \left(  \frac{1}{\tilde{b}_h(t)} -1 \right) u(t+ \sqrt{h} , x) + M \sqrt{h} \leq \left( \frac{1}{\tilde{b}_h(t)} - 1 \right)  + M \sqrt{h}.
	\end{equation}
	As well as,  for all $t\geq 0$ and $x\in \R$, one has
	\begin{equation}
		u(t,x) -u(t,x-h) \leq \left(  \frac{1}{\tilde{b}_h(t)} -1 \right) u(t+ \sqrt{h} , x-h) + M \sqrt{h} \leq \left( \frac{1}{\tilde{b}_h(t)} - 1 \right)  + M \sqrt{h}.
	\end{equation}
	Combined with \eqref{uch1} and \eqref{uch2}, this ensures that $u$ is uniformly continuous on $[0,\infty)\times\R$ and completes the proof of the lemma. 
	
\end{proof}

\begin{remark} \label{remark-v0}
	Here we point out that problem \eqref{Pb-dj2} is invariant with respect to spatial translation, so that spatial shift on the initial data $v_0(\cdot)$, induces the same spatial shift on the solution and does not change the uniform continuity on $[0,\infty)\times\R$. 
\end{remark}

\subsection{Proof of Theorem \ref{THEO_OUT_right-dj2} }
In this subsection, we construct a suitable exponentially decaying super-solution and prove Theorem \ref{THEO_OUT_right-dj2}.

\begin{proof}[Proof of Theorem \ref{THEO_OUT_right-dj2}]
	For each given $\lambda>0$ and sufficiently large $A>0$, let us firstly construct the following function
	\begin{equation*}
		\overline{u}(t,x):=
		\begin{cases}
			A e^{-\lambda_r^* \left(x- \int_{0}^{t} c(\lambda_r^*)(s) \d s \right)}, &\text{ if } \lambda \geq \lambda_r^*, \\
			A e^{-\lambda \left(x- \int_{0}^{t} c(\lambda)(s) \d s \right) }, & \text{ if } 0<\lambda< \lambda_r^*. 
		\end{cases}
	\end{equation*}
	Here we let $A>0$ large enough such that $\overline{u}(0,\cdot) \geq u_0 (\cdot)$ and recall that the speed function $t\mapsto c(\lambda)(t)$ is defined in \eqref{DEF-c}.
	
	Since $f(t,u) \leq \mu(t)$ for all $t\geq 0$ and $u\in [0, 1]$, then one readily obtains that $\overline{u}$ is  super-solution of \eqref{Pb-dj2}. So that the comparison principle implies that 
	$$\lim\limits_{t\to \infty} \sup_{x\geq \int_{0}^{t} c^{+} (\lambda)(s) \d s + \eta t} u(t,x) \leq \lim\limits_{t\to \infty} \sup_{x\geq \int_{0}^{t} c^{+} (\lambda)(s) \d s + \eta t}  \overline{u} (t,x) =0, \;\; \forall \eta >0.  $$
	
	This completes the proof of the upper estimate as stated in Theorem \ref{THEO_OUT_right-dj2}. 
	
\end{proof}

\subsection{Proof of Theorem \ref{THEO_IN_right-dj2} }
In this section we first discuss some properties of the solution of the following autonomous Fisher-KPP equation:
\begin{equation} \label{auto-kpp}
	\partial_t u(t,x) = \int_\R k(y) u(t,x-y) \d y - \bar{k} u(t,x) + u(t,x)(m- b u(t,x)), \;  t\geq 0, x\in\R.
\end{equation}
Here $k(\cdot)$ is a given symmetric kernel as defined in Remark \ref{remark-k}, $\bar{k}= \int_\R k(y) \d y>0$ while $m$ and $b$ are given positive constants.

Define 
$$
c_0:= \inf_{\lambda >0} \frac{ \int_\R k(y) e^{\lambda y} \d y - \bar{k} + m}{\lambda}.
$$
Note that $c_0 >0$ since $k(\cdot)$ is a symmetric function (see also \cite{xu2021spatial} where the sign of the (right and left) wave speed is investigated).
Next our first important result reads as follows.
\begin{lemma}\label{LEM-auto}
	Let $u=u(t,x)$ be the solution of \eqref{auto-kpp} supplemented with a continuous initial data $ 0 \leq u_0(\cdot) \leq \frac{m}{b} $ and $u_0 \not \equiv 0$ with compact support. Let us furthermore assume that $u$ is uniformly continuous for all $t\geq 0$, $x\in \R$. Then one has
	\begin{equation*}
		\lim\limits_{t\to \infty} \sup_{ |x| \leq ct} \left|\frac{m}{b}-u(t,x)\right|=0, \; \forall c\in [0,c_0).
	\end{equation*}
\end{lemma}

\begin{remark}
	For the kernel function with ${\rm supp} (k) = \R$ and without the uniform continuity assumption, the above propagating behaviour is already known. We refer to \cite[Theorem 3.2]{lutscher2005effect}. 
	For  the reader convenience, we give a short  proof of  Lemma \ref{LEM-auto}, with the help of Theorem 3.3 in \cite{xu2021spatial} and the additional regularity assumption of solution.
	%The skill will be used similarly in the proof of Theorem \ref{THEO_IN_right-dj2}.
\end{remark}

\begin{proof}
	Let $c\in [0, c_0)$ be given and fixed. To prove the lemma let us argue by contradiction by assuming that there exists a sequence $(t_n, x_n)$ and $|x_n| \leq ct_n$ such that 
	$$\limsup\limits_{n\to \infty} u(t_n,x_n) < \frac{m}{b}. $$
	
	Denote for $n\geq 0$ the sequence of functions $u_n$ by $u_n(t,x):= u(t+t_n, x+x_n)$. Since $u=u(t,x)$ is uniformly continuous on $[0,\infty)\times\R$ and $ 0\leq u \leq \frac{m}{b}$, then Arzel\`{a}-Ascoli theorem applies and ensures that as $n\to\infty$ one has
	$u_n(t,x)\to u_\infty(t,x)$ locally uniformly for $(t,x) \in \R^2$, for some function $u_\infty=u_\infty(t,x)$ defined in $\R^2$ and such that $u_\infty(0, 0) <\frac{m}{b}$. 
	
	Now fix $c'\in (c,c_0)$. Recall that Theorem 3.3 in \cite{xu2021spatial} ensures that there exists some constant $q_{c'} \in \left(0, \frac{m}{b}\right]$ such that
	$$ 
	\liminf_{t\to\infty}\inf_{ |x| \leq c't} u(t,x) \geq q_{c'}. 
	$$
	Hence there exists $T>0$ such that
	$$ 
	\inf_{ |x| \leq c't} u(t,x) \geq q_{c'}/2,\;\forall t\geq T. 
	$$
	This implies that for all $n\geq 0$ and $t\in\R$ such that $t+t_n\geq T$ one has
	$$
	\inf_{|x+x_n|\leq c'(t+t_n)} u(t+t_n, x+x_n) \geq q_{c'}/2.
	$$ 
	Since one has $ |x_n| \leq c t_n $ for all $n\geq 0$, this implies that for all $n\geq 0$ and $t\in\R$ with $t+t_n\geq T$:
	$$
	\inf_{|x|\leq (c'-c)t_n + c't} u(t+t_n, x+x_n) \geq q_{c'}/2. 
	$$
	Finally since $c'>c$ and $t_n \to \infty$ as $ n \to \infty $, then one has  $u_\infty(t,x) \geq q_{c'}/2>0$ for all $ (t,x) \in \R^2 $.

	Next, we consider $U=U(t)$ with $ U(0)= q_{c'}/2>0$ the solution of the ODE 
	\begin{equation*}
		U'(t) = U(t) \left( m- b U(t)\right),  \forall t\geq 0.
	\end{equation*}
	Since $u_\infty(s,x) \geq q_{c'}/2 $ for all $(s, x) \in \R^2$, then comparison principle implies that 
	$$u_\infty(t+s, x) \geq U(t), \forall t\geq 0, s\in \R, x\in \R.$$
	So that 
	$$ u_\infty(0, 0) \geq U(t), \forall t\geq 0. $$
	On the other hand, since $U(0)>0$, one gets $U(t) \to \frac{m}{b} $ as $t\to \infty$. Hence this yields $u_\infty(0,0) \geq \frac{m}{b}$, a contradiction with $u_\infty(0,0) < \frac{m}{b}$, which completes the proof. 
\end{proof}

Now we apply the key lemma to prove our inner propagation result Theorem \ref{THEO_IN_right-dj2}.

\begin{proof}[Proof of Theorem \ref{THEO_IN_right-dj2} (i)]
	Here we assume that the initial data $u_0$ has a fast decay rate and we aim at proving that
	$$
	\lim\limits_{t\to \infty} \sup_{x\in [0, ct] } | 1 - u(t,x)| =0, \; \forall c\in (0, c_r^*).
	$$
	One can construct a initial data $v_0$ alike in Lemma \ref{LEM-uc}, through choosing proper parameter and spatial shifting (see Remark \ref{remark-v0}) such that $v_0(x) \leq u_0(x)$ for all $x\in \R$. Let $v(t,x)$ be the solution of \eqref{sp} with initial data $v_0$. 
	 Lemma \ref{LEM-uc} ensures that $v(t,x)$ is uniformly continuous for all $t\geq 0, x\in \R$.  Since  $v_0 (\cdot ) \leq u_0(\cdot) $, then the comparison principle implies that $v(t,x)\leq u(t, x)$ for all $t\geq 0, x\in\R$. Note that $u(t,x)\leq 1$, it is sufficiently to prove that
	$$\lim_{ t \to \infty} \inf_{ x\in [0, ct] } v(t,x) = 1,  \; \forall c \in (0, c_r^*). $$
	Firstly, let us prove that 
	$$\liminf_{ t \to \infty} \inf_{  x \in [0, ct] } v(t,x)>0, \; \forall c \in (0, c_r^*).$$
	To do this, for all $B, R>0$, $\gamma\in\R$, we define $c_{R,B}(\gamma)$ by 
	\begin{equation}\label{def-c(t)-dj2}
		c_{R,B}(\gamma):=\frac{2R}{\pi}\int_{-B}^B K(z)\e^{\gamma z}\sin(\frac{\pi z}{2R})\d z.
	\end{equation} 
	Note that $ \gamma \mapsto c_{R,B}(\gamma) $ is continuous and recalling \eqref{eq-c*-dj2} one has 
	$$\lim\limits_{\gamma \to \lambda_r^*} \lim\limits_{\substack{R\to \infty \\ B\to \infty} }  c_{R,B}(\gamma)  = c_r^*. $$
	So for each $c' \in (c, c_r^*)$, one can choose proper $\gamma$ close to $\lambda_r^*$  such that for $R, B>0$ large enough,
	\begin{equation*}
		c' \leq c_{R,B}(\gamma). 
	\end{equation*}
	Then for all $ \frac{c}{c'}< k < 1$, 
	$$ \frac {c t}{k} \leq X(t):= c_{R, B}(\gamma) t.$$
	
	Now, we apply Lemma \ref{LEM-key-dj2} to show that 
	$$ \liminf_{ t \to+\infty} \inf_{0 \leq x \leq kX(t)} v(t,x)>0.$$
	Note that $t\mapsto X(t)$ is continuous for $t\geq 0$, and Lemma \ref{LEM-uc} ensures that  $v=v(t,x)$ is uniformly continuous for all $t\geq 0, x\in \R$.	We only need to check that $v=v(t,x)$ satisfies the conditions $(H1)-(H3)$ in Lemma \ref{LEM-key-dj2}.
	
	To show $(H1)$, recalling \eqref{lower-k} and \eqref{lip-dj2}, one may observe that $v=v(t,x)$ satisfies
	\begin{equation*}
		\partial_t v(t,x) \geq \int_{\R} k(y) v(t,x-y) \d y -\overline{K} v(t,x) + v(t,x) \left( \mu(t) -C v(t,x) \right).
	\end{equation*}
	Recalling Assumption \ref{ASS2-dj2} $(f4)$ and Lemma \ref{LEM-least}, there exists $a\in W^{1,\infty}(0,\infty)$ such that $\mu(t) -\overline{K} + a'(t)\geq 0 $ for all $t\geq 0$. Setting $w(t,x):=e^{a(t)} v(t,x)$ so that $w$ satisfies
	\begin{equation*}
		\begin{split}
			\partial_t w(t,x) &\geq \int_\R k(y) w(t,x-y) \d y - \bar{k} w(t,x) \\
			&\quad + w(t,x) \left( \bar{k} + \mu(t)-\overline{K} + a'(t) -C e^{-a(t)} w(t,x) \right) \\
			&\geq  \int_\R k(y) w(t,x-y) \d y - \bar{k} w(t,x) + w(t,x) \left( m -C e^{\|a\|_\infty} w(t,x) \right),	
		\end{split}
	\end{equation*}
	where $m:=\inf\limits_{t\geq 0} \left( \bar{k} +\mu(t) -\overline{K} + a'(t)\right)\geq \bar{k}>0$. 
	Now we consider $\underline{w}=\underline{w}(t,x)$ the solution of following equation
	\begin{equation}\label{eq-w-dj2}
		\partial_t \underline{w}(t,x)= k* \underline{w} (t,x) -\bar{k} \underline{w}(t,x) +  \underline{w}(t,x) \left( m -  C e^{\|a\|_\infty } \underline{w} (t,x) \right).
	\end{equation} 
	supplemented with the initial data $\underline{ w} (0,x) = e^{-\|a\|_\infty} v_0(x)$. Thus note that  one has $\underline{w}(0,x) \leq w(0,x)$ for all $x\in \R$ and the comparison principle implies that 
	$$w(t,x)= e^{a(t)}v(t,x) \geq \underline{w}(t,x), \;\; \forall t\geq 0, x\in\R. $$
	Lemma \ref{LEM-auto} implies that there exists $\tilde{c}>0$ such that 
	\begin{equation} \label{eq-limw}
		\lim_{t\to \infty} \sup_{ |x| \leq ct } \left| \underline{w}(t,x) - \frac{m}{Ce^{\|a\|_\infty}} \right| =0, \; \forall c\in (0, \tilde{c}).
	\end{equation}
	Since $a\in W^{1, \infty}(0, \infty ) $, we end-up with
	$$\liminf\limits_{t\to\infty} v(t,0) \geq \lim\limits_{t\to \infty} e^{-\|a\|_\infty} \underline{w}(t,0) = \frac{m}{Ce^{ 2 \|a\|_\infty} } >0, $$
	and $(H1)$ is fulfilled.

	Next we verify assumption $(H2)$. Recall that for all $\tilde{v} \in \omega (v) \setminus \{0\}$, there exist sequences $(t_n)_n$ with $t_n\to\infty$ and $(x_n)_n$ such that $\tilde{v}(t,x)= \lim\limits_{n\to \infty} v(t+t_n, x+x_n)$ where this limit holds locally uniformly for $(t,x)\in\R^2$. 
	As in the proof of Claim \ref{claim_dichotomy-dj2}, such a function $\tilde{v}$ satisfies 
	\begin{equation*}
		\partial_t \tilde{v}(t,x)\geq \int_\R k(y) \tilde{v}(t,x-y) \d y +  \tilde{v} (t,x) (\tilde{\mu}(t)- \overline{K}-C \tilde{v}(t,x)), \; \forall (t,x) \in\R^2,
	\end{equation*}
	where $k(y)$ is defined in \eqref{lower-k} and $\tilde{\mu}=\tilde{\mu}(t)\in L^\infty(\R)$ is a weak star limit of some shifted function $\mu(t_n+\cdot)$. 
	Similar to Definition \ref{def_least} and Lemma \ref{LEM-least}, one can define the least mean of $\tilde{\mu}$ over $\R$ as 
	$$ \lfloor \tilde{\mu} \rfloor = \lim_{ T \to \infty} \inf_{s\in\R} \frac{1}{T} \int_{0}^{T} \tilde{\mu} (t+s ) \d t.$$
	Also, the least mean of $\tilde{\mu} $ satisfies 
	$$   \lfloor \tilde{\mu} \rfloor= \sup_{a\in W^{1, \infty}(\R) } \inf_{ t\in\R } (a'+\tilde{\mu}) (t).$$
	Assumption \ref{ASS2-dj2} $(f4)$ implies that $\lfloor \tilde{\mu} \rfloor \geq \overline{K}$ and the same argument as above yields 
	$$\liminf\limits_{t\to \infty} \tilde{v}(t,0) \geq \frac{m}{Ce^{2\|b\|_\infty} }>0,$$
	where $b\in W^{1,\infty}(\R)$ such that $ \tilde{\mu}(t)- \overline{K} + b'(t) \geq 0$ for all $t\in \R$.
	Hence the condition $(H2)$ is satisfied.
	
	Before proving $(H3)$, we state a lemma related to a compactly supported sub-solution of \eqref{Pb-dj2}. Since \eqref{sp} is a special case of \eqref{Pb-dj2}, one can construct the similar sub-solution of \eqref{sp}. The following lemma can be proved similarly to Lemma 6.1 in \cite{ducrot2021generalized}.  So that the proof is omitted.
	\begin{lemma} \label{LEM-sub}
		Let Assumption \ref{ASS1-dj2}, \ref{ASS2-dj2} and \ref{ASS3} be satisfied. Let $\gamma\in (0,\lambda_r^*)$ be given. Then there exist $B_0>0$ large enough and $\theta_0>0$ such that for all $B>B_0$ there exists $R_0=R_0(B)>0$ large enough enjoying the following properties:	for all $B>B_0$ and $R>\max(R_0(B),B)$, there exists some function $a\in W^{1,\infty}(0,\infty)$ such that the function
		\begin{equation*}
			u_{R,B}(t,x)=\begin{cases} e^{a(t)} e^{-\gamma x}\cos (\frac{\pi x}{2R}) &\text{ if $t\geq 0$ and $x\in [-R,R]$},\\ 
				0 &\text{ else},
			\end{cases}
		\end{equation*}
		satisfies, for all $\theta\leq \theta_0$, for all $x\in [-R,R]$ and for any $t\geq 0$,
		\begin{equation*}
			\partial_t  u(t,x) -c_{R,B}(\gamma) \partial_x u(t,x)\leq \int_{\R} K(x-y)u(t,y)\d y+\left(\mu(t)-\theta-\overline K \right)u(t,x).
		\end{equation*}
		Herein the speed $c_{R,B}(\gamma)$ is defined in \eqref{def-c(t)-dj2}. Furthermore, let 
		$$\underline{u}(t,x):= \eta u_{R,B} (t,x -X(t)),$$
		where $X(t)= c_{R, B}(\gamma) t$  and $\eta>0$ small enough, then $\underline{u}(t,x)$ is the sub-solution of \eqref{Pb-dj2}. 
	\end{lemma}

	Now with the help of  Lemma \ref{LEM-sub} and the comparison principle, one can choose $\eta >0$ small enough such that $\underline{u}(0,x) \leq v_0(x)$ and therefore one has 
	$$\liminf\limits_{t\to \infty} v(t, X(t)) \geq  \liminf\limits_{ t\to \infty} \underline{u}(t, X(t))= \liminf\limits_{ t \to \infty} \eta u_{R, B}(t, 0)>0,$$
	which ensures that $(H3)$ is satisfied. 
	
	As a conclusion all the conditions of  Lemma \ref{LEM-key-dj2} are satisfied  and this yields  
	$$\liminf\limits_{ t \to \infty} \inf_{0 \leq x \leq kX(t)} v(t,x)>0. $$ 
	So that
	\begin{equation} \label{eq-po-dj2}
		\liminf\limits_{ t \to \infty} \inf_{ 0 \leq x \leq c t} v(t,x)>0, \; \forall c\in (0, c_r^*).
	\end{equation}
	
	Finally, let us prove that  
	$$\liminf\limits_{ t \to \infty} \inf_{ 0 \leq x \leq c t} v(t,x)=1,  \; \forall c\in (0, c_r^*). $$
	To do this, note that combining \eqref{eq-limw} and \eqref{eq-po-dj2} yields 
	$$ \liminf\limits_{ t \to \infty} \inf_{ -c_1 t \leq x \leq c t } v(t,x)>0, \forall  0<c_1<\tilde{c}, \forall c\in (0, c^{*}_r) .$$
	By the similar analysis to the proof of Lemma \ref{LEM-auto}, one could show that the above limit is equal to $1$. Hence the proof is completed.
\end{proof}

Next we prove Theorem \ref{THEO_IN_right-dj2} $(ii)$.
Firstly, we state a lemma about a sub-solution of \eqref{Pb-dj2}. One can also construct  the similar sub-solution for \eqref{sp}.
\begin{lemma}\label{LEM-sub2}
	Let Assumption \ref{ASS1-dj2}, \ref{ASS2-dj2} and \ref{ASS3} be satisfied. For each given $\lambda\in (0, \lambda_r^*)$, 
	define that 
	\begin{equation}\label{def-sub-dj2}
		\varphi(t,x) = e^{-\lambda (x+a(t))} - e^{-\lambda a(t)+B_0(t)+B_1} e^{-(\lambda + h)x}, \;\; t\geq 0, x\in\R, 
	\end{equation}
	where $a, B_0\in W^{1,\infty}(0,\infty)$, $B_1>0$ and $0<h<\min\left\{\lambda, \sigma (K) -\lambda \right\}$. Then 
	$$\underline{\phi}(t,x):=\max \left \{ 0, \;  \varphi\left(t,x-\int_{0}^{t}c_{\lambda, a} (s) \d s \right) \right \}$$ 
	is the subsolution of \eqref{Pb-dj2}.
\end{lemma}
\begin{remark}
	Note that $\varphi(t,x)$  is positive when 
	$$x>  \frac{\|B_0(t) \|_\infty + B_1}{h}.$$ We  point out  that 
	this lemma can be proved similarly to  \cite[Theorem 2.9]{ducrot2021generalized}. So we omit the proof.
\end{remark}

\begin{proof}[Proof of Theorem \ref{THEO_IN_right-dj2}(ii)] 
	As proof of Theorem \ref{THEO_IN_right-dj2} $(i)$, we can construct $v_0(x)$ alike in Lemma \ref{LEM-Exp}, through choosing proper parameter and spatial shifting (see Remark \ref{remark-v0}) such that $v_0(x) \leq u_0(x)$ for all $x\in \R$.
	Let $v(t,x)$ be the solution of \eqref{sp} equipped with initial data $v_0$. Lemma \ref{LEM-Exp} ensures that $v(t,x)$ is uniformly continuous for all $t\geq 0, x\in \R$.
	
	Recalling \eqref{DEF-c} and \eqref{DEF-ca}, for each given $\lambda \in (0, \lambda_r^*)$ and for all $c< c'< \lfloor c(\lambda) \rfloor $, one can choose a proper function $ a \in W^{1,\infty}(0,+\infty)$ such that
	$$ c' < c_{\lambda, a} (t), \; \forall t\geq 0. $$ 
	Then we define 
	$$X(t):= \int_{0}^{t}c_{\lambda,a}(s) \d s + P,$$
	where $P> \frac{\|B_0(t)\|_\infty +B_1}{h}>0$ and $B_0(\cdot)$, $B_1$ and $h$ are given in Lemma \ref{LEM-sub2}. Note that for all 
	$\frac{c}{c'}< k < 1$, 
	$$ c t \leq k X(t).$$
	
	Next it is sufficiently to apply key Lemma \ref{LEM-key-dj2} to show that 
	$$\liminf_{ t \to \infty} \inf_{ 0 \leq x\leq k X(t)} v(t, x)>0.$$
	Note that for exponential decay initial data $v_0$ on the right-hand side, that is $x\gg 1$, one can construct an initial data $\underline{v}_0$ alike in Lemma \ref{LEM-uc} with compact support such that $\underline{v}_0 \leq v_0$. Then  comparison principle implies  that $(H1)$ and $(H2)$ hold.  
	To verify the condition $(H3)$, by Lemma \ref{LEM-sub2} and comparison principle, one has
	$$ \liminf_{ t \to \infty} v(t, X(t)) \geq \liminf\limits_{t\to \infty} \underline{\phi}(t,X(t))= \liminf\limits_{t\to\infty} \varphi(t,P)>0. $$
	So $(H3)$ is satisfied. Hence the key Lemma \ref{LEM-key-dj2} ensures that 
	$$\liminf_{ t \to \infty} \inf_{ 0 \leq x\leq k X(t)} v(t, x)>0.$$ 
	Then one has
	$$\liminf_{ t \to \infty} \inf_{ 0 \leq x\leq ct } v(t, x)>0, \; \forall 0<c<\lfloor c(\lambda) \rfloor.$$
	Similarly to the proof of  Theorem \ref{THEO_IN_right-dj2} (i), one can show that 
	$$\lim_{ t \to \infty} \sup_{ x\in [0, ct] } |u(t, x) - 1 |=0, \; \forall 0<c<\lfloor c(\lambda) \rfloor.$$
	The proof is completed.
\end{proof}

Finally, we prove Corollary \ref{CORO_IN_right}. 

\begin{proof}[Proof of Corollary \ref{CORO_IN_right}]
	Recalling $H>0$ given in  Remark \ref{remark-f-dj2}, let us consider 
	\begin{equation}\label{sp1}
		\partial_{t} v(t,x) = \int_\R K(y) v(t,x-y) \d y - \overline{K} v(t,x) + \mu(t) v(t,x)\left( 1 - Hv(t,x) \right), \; t\geq 0, x\in \R.
	\end{equation}
	By the same analysis, one can obtain that the similar result for \eqref{sp1} as in Theorem \ref{THEO_IN_right-dj2}. For the reader convenience, we state it in the following. 
	
	Let $v=v(t,x)$ be the solution of \eqref{sp1} equipped with a continuous initial data $u_0$, with $ 0\leq u_0\leq 1$ and  $u_0 \not \equiv 0$. 
	Then the following inner spreading occurs:
	\begin{enumerate}
		\item[(i)] (fast exponential decay) If $u_0(x)=O(e^{-\lambda x})$ as $x\to\infty$ for some $\lambda\geq \lambda_r^*$ then one has
		$$
	\color{red}	\lim_{t\to\infty} \sup_{x\in [0, ct]}  \left| v(t,x)- \frac{1}{H} \right| =0,\;\forall c\in (0,c_r^*);
		$$
		\item[(ii)] (slow exponential decay) If $\displaystyle \liminf_{x\to\infty} e^{\lambda x}u_0(x) > 0 $ for some $\lambda\in (0,\lambda_r^*)$ then one has 
		$$
		\color{red} \lim_{t\to\infty} \sup_{x\in [0, ct]}   \left| v(t,x)- \frac{1}{H} \right| = 0,\;\forall c\in \left(0,\lfloor c(\lambda)\rfloor\right).
		$$
	\end{enumerate}
	Denote that  $u(t,x)$ is a solution of \eqref{Pb-dj2} 
	equipped with initial data $u_0$. Recall \eqref{lip-dj2} so that $v(t,x)$ is the sub-solution of \eqref{Pb-dj2}. Then comparison principle implies that $u(t,x) \geq v(t,x)$ for all $t\geq 0, x\in \R$. Hence the conclusion is proved.
\end{proof}

\vspace{3ex}

\noindent {\bf Acknowledgment:} We are very grateful to  two anonymous referees for  their careful reading and helpful comments which led to improvements of our original manuscript.  The second  author Z. Jin would like to acknowledge the r\'egion Normandie for the financial support of his PhD thesis.

\begin{small}
	\bibliography{az}

\begin{thebibliography}{10}

\bibitem{abi2021asymptotic}
L.~Abi~Rizk, J.-B. Burie, and A.~Ducrot.
\newblock Asymptotic speed of spread for a nonlocal evolutionary-epidemic
  system.
\newblock {\em Discrete \& Continuous Dynamical Systems}, 41(10):4959, 2021.

\bibitem{alfaro2017fujita}
M.~Alfaro.
\newblock Fujita blow up phenomena and hair trigger effect: the role of
  dispersal tails.
\newblock {\em Annales de l'Institut Henri Poincar{\'e} C}, 34(5):1309--1327,
  2017.

\bibitem{ambrosio2021generalized}
B.~Ambrosio, A.~Ducrot, and S.~Ruan.
\newblock Generalized traveling waves for time-dependent reaction--diffusion
  systems.
\newblock {\em Mathematische Annalen}, 381(1):1--27, 2021.

\bibitem{aronson1978multidimensional}
D.~G. Aronson and H.~F. Weinberger.
\newblock Multidimensional nonlinear diffusion arising in population genetics.
\newblock {\em Advances in Mathematics}, 30(1):33--76, 1978.

\bibitem{bao2018spreading}
X.~Bao, W.-T. Li, W.~Shen, and Z.-C. Wang.
\newblock Spreading speeds and linear determinacy of time dependent diffusive
  cooperative/competitive systems.
\newblock {\em Journal of Differential Equations}, 265(7):3048--3091, 2018.

\bibitem{berestycki2016persistence}
H.~Berestycki, J.~Coville, and H.-H. Vo.
\newblock Persistence criteria for populations with non-local dispersion.
\newblock {\em Journal of Mathematical Biology}, 72(7):1693--1745, 2016.

\bibitem{berestycki2008asymptotic}
H.~Berestycki, F.~Hamel, and G.~Nadin.
\newblock Asymptotic spreading in heterogeneous diffusive excitable media.
\newblock {\em Journal of Functional Analysis}, 255(9):2146--2189, 2008.

\bibitem{berestycki2005speed}
H.~Berestycki, F.~Hamel, and N.~Nadirashvili.
\newblock The speed of propagation for {KPP} type problems. {I:} periodic
  framework.
\newblock {\em Journal of the European Mathematical Society}, 7(2):173--213,
  2005.

\bibitem{berestycki2019asymptotic}
H.~Berestycki and G.~Nadin.
\newblock Asymptotic spreading for general heterogeneous {Fisher-KPP} type
  equations.
\newblock {\em Memoirs of the American Mathematical Society}, 2019.

\bibitem{cabre2013influence}
X.~Cabr{\'e} and J.-M. Roquejoffre.
\newblock The influence of fractional diffusion in {Fisher-KPP} equations.
\newblock {\em Communications in Mathematical Physics}, 320(3):679--722, 2013.

\bibitem{coville2013pulsating}
J.~Coville, J.~D{\'a}vila, and S.~Mart{\'\i}nez.
\newblock Pulsating fronts for nonlocal dispersion and {KPP} nonlinearity.
\newblock {\em Annales de L'Institut Henri Poincare Section (C) Non Linear
  Analysis}, 30(2):179--223, 2013.

\bibitem{coville2005propagation}
J.~Coville and L.~Dupaigne.
\newblock Propagation speed of travelling fronts in non local
  reaction--diffusion equations.
\newblock {\em Nonlinear Analysis: Theory, Methods \& Applications},
  60(5):797--819, 2005.

\bibitem{diekmann1979run}
O.~Diekmann.
\newblock Run for your life. {A} note on the asymptotic speed of propagation of
  an epidemic.
\newblock {\em Journal of Differential Equations}, 33(1):58--73, 1979.

\bibitem{ducrot2021asymptotic}
A.~Ducrot, T.~Giletti, J.-S. Guo, and M.~Shimojo.
\newblock Asymptotic spreading speeds for a predator--prey system with two
  predators and one prey.
\newblock {\em Nonlinearity}, 34(2):669, 2021.

\bibitem{ducrot2019spreading}
A.~Ducrot, T.~Giletti, and H.~Matano.
\newblock Spreading speeds for multidimensional reaction--diffusion systems of
  the prey--predator type.
\newblock {\em Calculus of Variations and Partial Differential Equations},
  58(4):1--34, 2019.

\bibitem{ducrot2021generalized}
A.~Ducrot and Z.~Jin.
\newblock Generalized travelling fronts for non-autonomous fisher-kpp equations
  with nonlocal diffusion.
\newblock {\em Annali di Matematica Pura ed Applicata (1923-)},
  201(4):1607--1638, 2022.

\bibitem{fang2017traveling}
J.~Fang, X.~Yu, and X.-Q. Zhao.
\newblock Traveling waves and spreading speeds for time--space periodic
  monotone systems.
\newblock {\em Journal of Functional Analysis}, 272(10):4222--4262, 2017.

\bibitem{finkelshtein2018hair}
D.~Finkelshtein and P.~Tkachov.
\newblock The hair-trigger effect for a class of nonlocal nonlinear equations.
\newblock {\em Nonlinearity}, 31(6):2442, 2018.

\bibitem{finkelshtein2019accelerated}
D.~Finkelshtein and P.~Tkachov.
\newblock Accelerated nonlocal nonsymmetric dispersion for monostable equations
  on the real line.
\newblock {\em Applicable Analysis}, 98(4):756--780, 2019.

\bibitem{fisher1937wave}
R.~A. Fisher.
\newblock The wave of advance of advantageous genes.
\newblock {\em Annals of eugenics}, 7(4):355--369, 1937.

\bibitem{garnier2011accelerating}
J.~Garnier.
\newblock Accelerating solutions in integro-differential equations.
\newblock {\em SIAM Journal on Mathematical Analysis}, 43(4):1955--1974, 2011.

\bibitem{girardin2019invasion}
L.~Girardin and K.-Y. Lam.
\newblock Invasion of open space by two competitors: spreading properties of
  monostable two-species competition-diffusion systems.
\newblock {\em Proceedings of the London Mathematical Society},
  119(5):1279--1335, 2019.

\bibitem{hale1989persistence}
J.~K. Hale and P.~Waltman.
\newblock Persistence in infinite-dimensional systems.
\newblock {\em SIAM Journal on Mathematical Analysis}, 20(2):388--395, 1989.

\bibitem{jin2012seasonal}
Y.~Jin and M.~A. Lewis.
\newblock Seasonal influences on population spread and persistence in streams:
  spreading speeds.
\newblock {\em Journal of Mathematical Biology}, 65(3):403--439, 2012.

\bibitem{jin2009spatial}
Y.~Jin and X.-Q. Zhao.
\newblock Spatial dynamics of a periodic population model with dispersal.
\newblock {\em Nonlinearity}, 22(5):1167, 2009.

\bibitem{kao2010random}
C.-Y. Kao, Y.~Lou, and W.~Shen.
\newblock Random dispersal vs. non-local dispersal.
\newblock {\em Discrete \& Continuous Dynamical Systems}, 26(2):551, 2010.

\bibitem{kolmogorov1937study}
A.~N. Kolmogorov, I.~G. Petrovsky, and N.~S. Piskunov.
\newblock {\'E}tude de l'\'equation de la diffusion avec croissance de la
  quantit\'e de mati\`ere et son application \`a un probl\`eme biologique.
\newblock {\em Bulletin Universit\'e d'Etat de Moscou}, 1(6):1--26, 1937.

\bibitem{li2010entire}
W.-T. Li, Y.-J. Sun, and Z.-C. Wang.
\newblock Entire solutions in the {Fisher-KPP} equation with nonlocal
  dispersal.
\newblock {\em Nonlinear Analysis: Real World Applications}, 11(4):2302--2313,
  2010.

\bibitem{liang2006spreading}
X.~Liang, Y.~Yi, and X.-Q. Zhao.
\newblock Spreading speeds and traveling waves for periodic evolution systems.
\newblock {\em Journal of Differential Equations}, 231(1):57--77, 2006.

\bibitem{liang2007asymptotic}
X.~Liang and X.-Q. Zhao.
\newblock Asymptotic speeds of spread and traveling waves for monotone
  semiflows with applications.
\newblock {\em Communications on Pure and Applied Mathematics: A Journal Issued
  by the Courant Institute of Mathematical Sciences}, 60(1):1--40, 2007.

\bibitem{liang2020jfa}
X.~Liang and T.~Zhou.
\newblock Spreading speeds of nonlocal {KPP} equations in almost periodic
  media.
\newblock {\em Journal of Functional Analysis}, 279(9):108723, 2020.

\bibitem{lim2016transition}
T.~S. Lim and A.~Zlato{\v{s}}.
\newblock Transition fronts for inhomogeneous {Fisher-KPP} reactions and
  non-local diffusion.
\newblock {\em Transactions of the American Mathematical Society},
  368(12):8615--8631, 2016.

\bibitem{lou2013recurrent}
B.~Lou, H.~Matano, and K.-I. Nakamura.
\newblock Recurrent traveling waves in a two-dimensional saw-toothed cylinder
  and their average speed.
\newblock {\em Journal of Differential Equations}, 255(10):3357--3411, 2013.

\bibitem{lutscher2005effect}
F.~Lutscher, E.~Pachepsky, and M.~A. Lewis.
\newblock The effect of dispersal patterns on stream populations.
\newblock {\em SIAM Review}, 47(4):749--772, 2005.

\bibitem{magal2005global}
P.~Magal and X.-Q. Zhao.
\newblock Global attractors and steady states for uniformly persistent
  dynamical systems.
\newblock {\em SIAM Journal on Mathematical Analysis}, 37(1):251--275, 2005.

\bibitem{matano2011large}
H.~Matano and M.~Nara.
\newblock Large time behavior of disturbed planar fronts in the {Allen--Cahn}
  equation.
\newblock {\em Journal of Differential Equations}, 251(12):3522--3557, 2011.

\bibitem{nadin2012propagation}
G.~Nadin and L.~Rossi.
\newblock Propagation phenomena for time heterogeneous {KPP}
  reaction--diffusion equations.
\newblock {\em Journal de Math{\'e}matiques Pures et Appliqu{\'e}es},
  98(6):633--653, 2012.

\bibitem{nadin2015transition}
G.~Nadin and L.~Rossi.
\newblock Transition waves for {Fisher--KPP} equations with general
  time-heterogeneous and space-periodic coefficients.
\newblock {\em Analysis \& PDE}, 8(6):1351--1377, 2015.

\bibitem{schumacher1980travelling}
K.~Schumacher.
\newblock Travelling-front solutions for integro-differential equations. {I}.
\newblock {\em Journal f{\"u}r die reine und angewandte Mathematik},
  1980(316):54--70, 1980.

\bibitem{shen2010variational}
W.~Shen.
\newblock Variational principle for spreading speeds and generalized
  propagating speeds in time almost periodic and space periodic {KPP} models.
\newblock {\em Transactions of the American Mathematical Society},
  362(10):5125--5168, 2010.

\bibitem{shen2016transition}
W.~Shen and Z.~Shen.
\newblock Transition fronts in nonlocal {Fisher-KPP} equations in time
  heterogeneous media.
\newblock {\em Communications on Pure \& Applied Analysis}, 15(4):1193, 2016.

\bibitem{shen2010spreading}
W.~Shen and A.~Zhang.
\newblock Spreading speeds for monostable equations with nonlocal dispersal in
  space periodic habitats.
\newblock {\em Journal of Differential Equations}, 249(4):747--795, 2010.

\bibitem{smith2011dynamical}
H.~L. Smith and H.~R. Thieme.
\newblock {\em Dynamical systems and population persistence}, volume 118.
\newblock American Mathematical Soc., 2011.

\bibitem{weinberger1982long}
H.~F. Weinberger.
\newblock Long-time behavior of a class of biological models.
\newblock {\em SIAM Journal on Mathematical Analysis}, 13(3):353--396, 1982.

\bibitem{weinberger2002spreading}
H.~F. Weinberger.
\newblock On spreading speeds and traveling waves for growth and migration
  models in a periodic habitat.
\newblock {\em Journal of Mathematical Biology}, 45(6):511--548, 2002.

\bibitem{xu2020spatial}
W.-B. Xu, W.-T. Li, and S.~Ruan.
\newblock Spatial propagation in an epidemic model with nonlocal diffusion: The
  influences of initial data and dispersals.
\newblock {\em Science China Mathematics}, 63(11):2177--2206, 2020.

\bibitem{xu2021spatial}
W.-B. Xu, W.-T. Li, and S.~Ruan.
\newblock Spatial propagation in nonlocal dispersal {Fisher-KPP} equations.
\newblock {\em Journal of Functional Analysis}, 280(10):108957, 2021.

\bibitem{zhang2020propagation}
G.-B. Zhang and X.-Q. Zhao.
\newblock Propagation phenomena for a two-species {Lotka--Volterra} strong
  competition system with nonlocal dispersal.
\newblock {\em Calculus of Variations and Partial Differential Equations},
  59(1):1--34, 2020.

\bibitem{zhao2003dynamical}
X.-Q. Zhao.
\newblock {\em Dynamical systems in population biology}, volume~16.
\newblock Springer, 2003.

\end{thebibliography}
	\bibliographystyle{abbrv}
\end{small}

\end{document}